\documentclass[12pt, a4paper,oneside]{amsart} 
\usepackage{amssymb, amsfonts, mathrsfs}
\usepackage[textwidth=1in,backgroundcolor=white,linecolor=white,bordercolor=white,textcolor=red]{todonotes}

\usepackage{glossaries}
\usepackage[bookmarks]{hyperref} 
\usepackage[maxbibnames=99]{biblatex}
\DeclareFieldFormat*{url}{}
\DeclareFieldFormat[unpublished]{url}{\mkbibacro{URL}\addcolon\space\url{#1}}
\DeclareFieldFormat*{urldate}{}
\DeclareFieldFormat[unpublished]{urldate}{\mkbibparens{\bibstring{urlseen}\space#1}}

\usepackage{dirtytalk}

\newcommand{\g}[1]{\gls{#1}}

\newglossaryentry{graphs} {
	name={\ensuremath{\bar{\Theta}}},
	description={the compact family of graphs}
}
\newglossaryentry{graphs_dual} {
	name={\ensuremath{\bar{\Theta^{*} }}},
	description={the compact family of dual graphs}
}
\newglossaryentry{incr} {
	name={\ensuremath{\mathcal{I}}},
	description={the space of increments}
}
\newglossaryentry{traj} {
	name={\ensuremath{\mathcal{T}}},
	description={the space of trajectories}
}

\usepackage{lmodern} 

\usepackage[textwidth=1in,backgroundcolor=white,linecolor=white,bordercolor=white,textcolor=red]{todonotes}

\usepackage{amsmath}
\usepackage{amsthm}
\usepackage{mathtools}  
\mathtoolsset{showonlyrefs}  

\usepackage{tkz-euclide}
\usepackage{tikz}
\usetikzlibrary{patterns}
\usepackage{thmtools, thm-restate}
\newtheorem{theorem}{Theorem}
\newtheorem{definition}{Definition}
\newtheorem*{theorem*}{Theorem}
\newtheorem{lemma}[theorem]{Lemma}
\newtheorem{propo}{Proposition}[section]
\theoremstyle{definition}

\theoremstyle{remark}
\newtheorem*{Remark}{Remark}
\theoremstyle{remark}

\newtheorem{corollary}{Corollary}

\renewcommand\P{\mathbb{P}}
\newcommand\C{\mathbb{C}}
\newcommand\R{\mathbb{R}}
\newcommand\N{\mathbb{N}}
\newcommand\D{\mathbb{D}}
\newcommand\Z{\mathbb{Z}}
\newcommand\HH{\mathbb{H}}

\newcommand\FF{\mathcal{F}}
\newcommand\PP{\mathcal{P}}

\newcommand\supp{\mathrm{supp}}

\newcommand\tr{\mathrm{tr}}
\newcommand\Vol{\mathrm{Vol}}

\usepackage{mathtools}
\DeclarePairedDelimiter\abs{\lvert}{\rvert}
\DeclarePairedDelimiter\norm{\lVert}{\rVert}

\title[Lyapunov exponents and stability]{Lyapunov exponents and stability properties of higher rank
representations}
\author{Florestan {Martin-Baillon}}

\begin{document}

\maketitle


\begin{abstract}
	For a holomorphic family
	$ (\rho_{\lambda}) $ 
	of representations
	$ \Gamma \to \text{SL}(d, \C) $,
	where $ \Gamma $ is a finitely
	generated group,
	we introduce the notion
	of proximal stability
	and show that it is equivalent
	to the pluriharmonicity
	of Lyapunov exponents
	of the family
	(defined using random walks).
\end{abstract}

\tableofcontents

\section{Introduction}
\label{sec:introduction}

There exists a deep analogy
between the theory of Kleinian
groups
and one-dimensional
complex dynamics.
This idea is often
refered to as
the \emph{Sullivan dictionary},
because Sullivan
explored with great
success this connection
(see \cite{sullivanQuasiconformalHomeomorphismsDynamics1985a},
\cite{sullivanQuasiconformalHomeomorphismsDynamics1985},
\cite{mcmullenQuasiconformalHomeomorphismsDynamics1998}
).
The dictionary
is a conceptual framework
whose goal
is to \say{translate}
some concepts and methods
of one theory into
corresponding notions
for the other theory.
For example,
a rational map
corresponds to
a Kleinian group,
and the Julia set
of the rational map
corresponds
to the limit set
of the Kleinian group.
One of the major achievement of this dictionary
has been the proof of the
\say{No wandering domain} theorem by Sullivan
\cite{sullivanQuasiconformalHomeomorphismsDynamics1985a}
(for rational maps)
using methods developed by Ahlfors
for his proof of the Finiteness theorem
\cite{ahlforsFinitelyGeneratedKleinian1964}
(for Kleinian groups).
A major common tool
is the use of quasiconformal
maps and deformations.

In a serie of works,
starting by
\cite{deroinRandomWalksKleinian2012},
Deroin and Dujardin
extended this dictionary.
They developped
the study
of \emph{bifurcation currents}
for Kleinian groups,
following the
work of
DeMarco
\cite{demarcoDynamicsRationalMaps2001}
who introduced bifurcation currents
for holomorphic family of rational maps
on the complex projective line.
Such bifurcation currents
live on the parameter
space and connect
the \emph{stability} of the family
and the potential theory
of a certain function,
the \emph{Lyapunov exponent}.
The main property
of the bifurcation current
is that its support
is precisely the bifurcation
locus (suitably defined
for family of Kleinain groups),
and this fact gives
informations about this locus.
For example, they
prove that
some dynamically defined
subvarieties equidistribute
toward the bifurcation current.
The goal of the present work
is to extend the theory
of bifurcation currents
to \emph{higher rank representations},
that is the higher dimensional
generalization of Kleinian groups,
still in parallel with
the theory of complex
dynamics in higher dimension.
First we present some of
the existing work
which develop
the theory of
higher rank representations
on one side
and
the theory of
higher dimensional
complex dynamics
on the other side.

The theory of higher
rank representations
is a body of work
aiming at generalizing
the theory of Fuchsian
and Kleinian groups
to subgroups
of Lie groups
of rank $ \ge 2 $.
It is sometime called
the
\emph{Higher Teichmüller Theory}.
It originated in the study
of homogeneous geometric structures
(see \cite{goldmanGeometricStructuresVarieties1988}).
The introduction of the Hitchin
component \cite{hitchinLieGroupsTeichmuller1992}
as a generalization of
the Teichmuller space,
which extends the
$ \text{PSL}(2, \R) $
theory to
$ \text{PSL}(n, \R) $
generated a wealth
a subsequent works
which explored
the properties of
this component.
In order to understand
the Hitchin representations
(representations that lie
in the Hitchin component),
Labourie 
\cite{labourieAnosovFlowsSurface2006}
introduced
the notion of
Anosov representations,
a concept that was greatly
studied and extended
(see \cite{guichardAnosovRepresentationsDomains2012},
\cite{gueritaudAnosovRepresentationsProper2017},
\cite{kapovichAnosovSubgroupsDynamical2017},
\cite{kapovichDynamicsFlagManifolds2018},
\cite{bochiAnosovRepresentationsDominated2019}
for example).
Anosov representations
are now believed
to be the \say{right}
higher rank generalization
of convex-cocompact
representations
into rank 1 Lie groups.
They have interesting dynamical
properties,
notably
dynamical stablity
in their limit sets,
and
form open components
of the space of representations.

High dimensional complex
dynamics is the study
of iteration of
rational maps
of $ \P^{d}( \C ) $.
A major difficulty
in the study of
family of rational maps
is that the main
tools in dimension one,
the $ \lambda $-lemma
and the quasiconformal
theory
is not available
in higher dimensions.
This difficulty
has been tackled
by the use
of \emph{pluripotential theory},
briefly the theory
of plurisubharmonic functions
and positive currents,
see \cite{dinhDynamicsSeveralComplex2010}
and the survey
\cite{dujardinBifurcationCurrentsEquidistribution2014} 
for an exposition about
bifurcation currents
and
the difficulties
of a generalization to higher dimension.
A recent work
of Berteloot, Bianchi
and Dupont
\cite{bertelootDynamicalStabilityLyapunov2018}
studies the
stability of family
of endomorphisms
of $ \P^{d} (\C) $
and proves a
generalization
of the results of
DeMarco:
there exists 
a bifurcation current
whose support
is the bifurcation locus.
leading to a richer theory.
For polynomial
automorphisms of
$ \C^2 $,
Dujardin-Lyubich
\cite{dujardinStabilityBifurcationsDissipative2014}
and
then
Dujardin-Berger
\cite{bergerStabilityHyperbolicityPolynomial2017}
introduce
another notion
of stability/bifurcation
and give various dynamical characterization
of stability.
Part of these works
is to define the
notions of stability
and bifurcation
in this context:
if in dimension one all notions of stability coincide,
this is not the case anymore in higher dimension,

When 
studying
dynamically interesting
bifurcation current
(as in
\cite{demarcoDynamicsRationalMaps2001},
\cite{deroinRandomWalksKleinian2012}
and \cite{bertelootDynamicalStabilityLyapunov2018}),
the bifurcation
current
is defined
as the $ dd^c $
of the Lyapunov exponent
(of the endomorphism or
the representation),
but it is always
a non-trivial task
to prove
that it correspond to
a useful notion
of stability,
see section
\ref{sub:comp}.

Let us now explain
the results of
the present work.
We state the results
informally,
and refer to the
section \ref{sec:prelim}
for the precise definitions.
We study holomorphic family
of representation
$ \rho_{\lambda}:\Gamma \to G $ 
where
$ \Gamma $ is a finitely generated
group,
$ G = SL(d, \C) $
and $ \lambda $
varies in a complex manifold
$ \Lambda $.
As in
\cite{deroinRandomWalksKleinian2012},
we define the
Lyapunov exponents
$ (\chi_{i}(\lambda))_{i=1,\dots,d} $
of the representation
$ \rho_{\lambda} $ 
using a random walk
on $ \Gamma $.
The bifurcation current
is defined as
$ T_{bif} =
dd^{c}( \chi_{1}(\lambda)
+ \chi_{d}(\lambda)) $.
Inspired by
\cite{bertelootDynamicalStabilityLyapunov2018},
we define the notion
of \emph{proximal stability}
for a representation:
a family of representations
$ (\rho_{\lambda}) $ is
\emph{proximally stable}
if for any
$ \gamma \in \Gamma $
and $ \lambda_{0} $ 
such that
$ \rho_{\lambda_{0}} (\gamma) $
is proximal,
then 
$ \rho_{\lambda} (\gamma) $
is proximal for every
$ \lambda $.
The stability locus
is the set of parameters
where this property holds
and the bifurcation locus
is the complement of
the stability locus.
Our main result is:
\begin{center}
	\emph{Under some assumptions
	      on the family of representations,
	      the support of the bifurcation
	      current is the
	      bifurcation locus.}
\end{center}
This statement is of the
same nature as in
\cite{demarcoDynamicsRationalMaps2001},
\cite{deroinRandomWalksKleinian2012}
and
\cite{bertelootDynamicalStabilityLyapunov2018}.
For a precise statement,
see \ref{sec:statement}.

\textbf{Acknowledgments.}
The author wants to thanks
his advisor Bertrand Deroin
for introducing him to this subject
and his constant support.
He also thanks Christophe Dupont
for generously sharing his time
to explain his work.

\section{Preliminaries}
\label{sec:prelim}

We collect here the necessary
definitions and some useful facts
and results
about random walks on abstract
and linear groups.
We also recall some notions
of complex analysis.
We fix a finitely generated group $\Gamma$,
a complex vector space $ V $
of dimension $ d+1 $ 
and $G = SL(V) = SL(d+1, \C)$.

\subsection{Properties of linear subgroups}
\label{sub:linear_action}

We describes some notions relative to
a subgroup (more generaly a subsemigroup)
of $ G $ acting on $ V $ and $ \P V $.

Let $ \Lambda \subset G $
a subsemigroup of $ G $.
We say that $ \Lambda $
is \emph{irreducible}
if there does not
exist a non-trivial subspace
$ W $
of $ V $ stabilized by
all elements of $ \Lambda $,
i.e.
for all $ \gamma \in \Lambda $,
$ g W = W $.
We say that is it
\emph{strongly irreducible}
if there doesn't
exist a finite number
of non-trivial subspaces
$ W_{1}, \dots, W_{r}  $
of $ V $
such that the reunion
$ \bigcup W_{i} $
is
stabilized by
all elements of $ \Lambda $,
i.e.
for all $ i $ 
and for $ \gamma \in \Lambda $,
there exists a $ j $
such that
$ g W_{i} = W_{j} $.
Remark that $ \Lambda $
is strongly irreducible
if all of its
finite index
subsemigroups are
irreducible.

The main
theme of this work
is to understand the
dynamical properties
of a subsemigroup.
These properties
are related
to the contraction of
a transformation
on the projective space.
Here are the main notions
we use.

An element $ g \in G $ is \emph{proximal} if it has
a unique eigenvalue (counting multiplicity)
of maximal modulus.
In this case there exists
a fixed point
$ Fix^{+} g \in \P V $
and a fixed hyperplane
$ Fix^{-} g \in \P V^{*} $
such that
the orbit of every point not in $ Fix^{-} g $
converges (exponentially fast)
to $ Fix^{+} g $.

The semigroup $ \Lambda $ 
is proximal
if there exists a sequence
$ g_{n} \in \Lambda $
and a sequence of scalars
$ a_{n} \in \C $ 
such that
$ a_{n} g_{n} $ 
converges to an endomorphism
of rank 1.
If $ \Lambda $ 
contains a proximal element,
then it is is proximal.
If $\Lambda$ is irreducible,
the converse is true
(see \cite[Lem. 4.1]{benoistRandomWalksReductive2016})
In this case we think of the action of $ \Lambda $
as contracting on $ \P V $, because
the action of a rank 1 endomorphism $ \pi $ on
$ \P V $ contracts everything
(away from $ \P (\ker \pi) $)
to the point $ \P (Im \pi) $.

When $ \Lambda $ is irreducible
and proximal, it admits
a well-defined \emph{limit set}
$ L(\Lambda) \subset \P V $.
It is the unique minimal
$ \Lambda $-invariant
subset of $ \P V $.
By definition, it is
the set of points $ x \in \P V $
of the form
$ x = Im(\pi) $ 
for
endomorphisms $ \pi $
of rank 1
obtained
as limits
of sequences
$ a_{n} g_{n} $ 
where
$ g_{n} \in \Lambda $
and
$ a_{n} \in \C $
(see \cite[Lem. 4.2]{benoistRandomWalksReductive2016}).
It is the set where the dynamic of $ \Lambda $ 
is concentrated.
It is the closure of the set
of attracting fixed points
of proximal elements of
$ \Lambda $.

For a linear representation
$ \rho : \Gamma \to G $,
we say that the representation
has the property $ P $
if its image $ \rho(\Gamma) $
has the property $ P $
(where $ P $ can be irreducible, proximal, etc.).

\subsection{Random walks and Lyapunov exponents}

We introduce the notion of random walks
on group.
We will see that the data of random walk
which generates a semigroup $ \Lambda $
allows us to quantify all the
qualitative properties of $ \Lambda $
introduced in the previous paragraph
\ref{sub:linear_action}.

\subsubsection{Random walk on the abstract group}
\label{ssub:rw_abstract}

Let us fix a symmetric generator set of $ \Gamma $
and the associated word distance
$ \abs{\cdot}_{\Gamma} $.
Let's fix a probability mesure
$ \mu $ on $ \Gamma $
such that the support
of $ \mu $ generates $ \Gamma $
as a semi-group.
An example of such a probability
is
\begin{equation}
	\mu
	=
	\frac{1}{\# S}
	\sum_{s \in S} 
	\delta_{s} 
\end{equation}
where $ S $ is
a symmetric generator set
for $ \Gamma $.

The (left) random walk on $ \Gamma $
induced by $ \mu $ is the
Markov chain with transistion
probability at 
$ \gamma $
given by
$ (\gamma^{-1})_{*} \mu $.
This means
that a trajectory for this process
starting at $ \gamma_{0} $ is given
by
$ (\gamma^{(n)})$
where
$ \gamma^{(n)} =
\gamma_{n} \cdots
\gamma_{0}
$ 
and
$ (\gamma_{i}) $ 
is a sequence of independent random
variables with law $ \mu $.
The law of the $n$-th step
of the random walk is $ \mu^{n} $,
the measure $ \mu $ convoluted
$ n $-times with itself.
We call $ (\gamma_{n}) $
the sequence of increments
and
$ (\gamma^{(n)}) $
the associated
trajectory.
We denote by
$ \g{incr} $ the set
$ \Gamma^{\N} $
and call it the
space of increments,
which we
equip with the measure
$ \bar{\mu} = \mu^{\otimes \N} $.
This set parametrizes
the trajectories
of the random walk.

We will assume that $ \mu $ 
satisfies an exponential moment,
that is there exists $ \alpha > 0 $
such that
\begin{equation}
	\int_{\Gamma}
	\exp ( \alpha
	\abs{\gamma}_{\Gamma}
	)
	d \mu (\gamma)
	< \infty .
\end{equation}
It is always satisfied
if the support of $ \mu $
is finite.
It implies the weaker
first moment condition:
\begin{equation}
	\int_{\Gamma}
	\abs{\gamma}_{\Gamma}
	d \mu (\gamma)
	< \infty .
\end{equation}

We will use
the reversed random walk,
defined by:
$ \check{\mu} (\gamma) = \mu (\gamma^{-1}) $.
Remark that it satisfies the same
moment conditions as $ \mu $.

\subsubsection{Linear random walks}
\label{ssub:rw_lin}

In our setting, linear random walks
arise
in the following way.
Let $ \mu $ be a probability
on $ \Gamma $ as in the previous
paragraph and
let $ \rho: \Gamma \to G $ be a linear representation.
It induces a random walk on $ G $
given by $ \mu_{\rho} := \rho_{*} \mu $,
supported on $ \rho(\Gamma) $.

Because $ \mu $ satisfies and exponential
moment, $ \rho_{*} \mu $ satisfies
a (linear) exponential moment:
\begin{equation}
	\int_{\Gamma}
	\norm{
		\rho(\gamma)
	}^{\alpha'}
	d \mu (\gamma)
	\le
	\int_{\Gamma}
	\exp
	(
	\alpha'
	C_{\rho}
	\abs{\gamma}_{\Gamma}
	)
	d \mu (\gamma)
	<
	\infty
\end{equation}
where
$ C_{\rho}
=
\log
\max_{s}
\norm{\rho(s)}
$
for $ s $ in the generating
set defining the distance on $ \Gamma $,
and $ \alpha' $ is such that
$ \alpha' C_{\rho} \le \alpha $.

The largest Lyapunov exponent associated
to this random walk is by definition
\begin{equation}
\label{eq:def_lyap}
	\chi ( \rho )
	:=
	\lim_{n \to \infty} 
	\frac{1}{n}
	\int_{\Gamma}
		\log
		\norm{ \rho ( \gamma ) }
		d \mu^{n} (\gamma)
	,
\end{equation}
where $ \norm{\cdot} $ is any norm
(and the limit does not depend on the norm).
The limit is well defined by subadditivity
of the sequence appearing
and by finiteness of the first moment.
We define the \emph{Lyapunov spectrum}
$ \chi_{1}, \dots, \chi_{d+1} $ inductively
by
\begin{equation}
	\chi_{1} ( \rho )
	+
	\dots
	+
	\chi_{i} (\rho)
	:=
	\chi ( \wedge^{i} \rho)
	,
\end{equation}
where $ \wedge^{i} \rho $ 
is the induced representation
on the $ i $-th exterior power
$ \wedge^{i} V $.
We have
$ \chi_{1} \ge \dots \ge \chi_{d+1} $
and
$ \chi_{1} + \dots + \chi_{d+1} = 1 $,
so in particular $ \chi_{1} \ge 0 $.
In this article, we will be mainly
interested about $ \chi_{1} $, $ \chi_{2} $
and $ \chi_{d+1} $.

The largest Lyapunox exponent
controls the exponential growth
rate of typical random matrix product.
More precisely, we have the
following \say{Law of Large Numbers}:
\begin{theorem}
	[Furstenberg-Kesten]
	For $ \bar{\mu} $-almost all
	trajectories $ (\gamma_{n}) \in \g{incr} $
	we have
	\begin{equation}
		\lim
		_{n \to \infty} 
		\frac{1}{n}
		\log
		\norm{
			\rho(
			\gamma_{n}
			\cdots
			\gamma_{1} 
			) 
		}
		=
		\chi_{1} (\rho)
		.
	\end{equation}
\end{theorem}
Of course, if $ \chi_{1} = 0 $
we don't get much information.
For example, if the random walk
generates a semigroup
supported on a bounded group,
it is always the case.
This is essentialy the only obstruction:
\begin{theorem}
	[Furstenberg]
	If $ \rho $ is
	strongly irreducible
	and unbounded
	then
	\begin{equation}
		\chi_{1}(\rho)
		> 0
		.
	\end{equation}
\end{theorem}
This theorem is emblematic
of the philosophy of this topic:
it improves a qualitative statement
(unboundedness) to a quantitative one:
almost every trajectory grows
exponentialy fast (at a given rate).
In the same spirit we have
\begin{theorem}
	[Furstenberg]
	If $\rho$ is strongly irreducible,
	then $ \rho $ is proximal
	if and only if
	$ \chi_{1} (\rho) > \chi_{2} (\rho) $.
\end{theorem}
Remark that $ \chi_{2} - \chi_{1} $
is the average exponential rate of contraction
of tangent vector in $ \P V $.

We want to be more precise and to
describe the behavior of random sequences
$ \rho(\gamma_{n} \cdots \gamma_{1}) v $ 
for $ v \in V $.
For that we will assume from now on that
\textbf{$ \rho $ is proximal and strongly irreducible}
(it is not the weakest assumption, but
it will be enough for the applications we
have in mind).

A measure $ \nu $ on $ \P V $
is $ \mu_{\rho}  $-\emph{stationary}
if $ \mu_{\rho}  * \nu = \nu $
where
$ 
\mu_{\rho}  * \nu =
\int_{\Gamma} 
\rho(\gamma)_{*} \nu
d\mu(\gamma)
$.
Such a measure
always exists
and
under the proximal
and strongly irreducible
assumption
it is unique
(see for example
\cite[Prop. 4.7]{benoistRandomWalksReductive2016}).
We denote it by $ \nu_{\rho} $.
Its support is precisely the limit set
$ L(\rho) $ of $ \rho(\Gamma) $.
It is a proper measure:
every linear subspace of $ \P V $
which is not $ \P V $ is given
zero mass.
The stationary measure describes
the distribution of the sequences
$ (\rho(\gamma_{n} \cdots \gamma_{1}) v) $;
more precisely, for $ \nu_{\rho} $-almost
every $ x \in \P V $,
the sequence
$ (\rho(\gamma_{n} \cdots \gamma_{1}) x) $
equidistributes according to $ \nu_{\rho} $.

The stationary measure allows us
to define the \emph{Furstenberg limit map}
$ \theta_{\rho} : \g{incr} \to \P V $,
a measurable map
which satisfies the following properties:
\begin{itemize}
	\item for
		$ \bar{\mu} $-almost all
		$ b = (\gamma_{n}) \in \g{incr} $
		we have
		\begin{equation}
			\rho(\gamma_{1} \cdots \gamma_{n})_{*} 
			\nu_{\rho} 
			\to_{n \to \infty} 
			\delta_{\theta_{\rho}(b)} 
		\end{equation}
	\item $ (\theta_{\rho })_{*} \bar{\mu}
		= \nu_{\rho} $ 
	\item for $ \bar{\mu} $-almost all
		$ b = (\gamma_{n}) \in \g{incr} $,
		any endomorphism $ \pi $
		which is a limit point
		of a sequence
		$ a_{n} \rho(\gamma_{1} \dots \gamma_{n}) $ 
		with $ a_{n} \in \C $
		is of rank 1 and satisfies
		$ Im(\pi) = \theta (b) $.
\end{itemize}
The first property is actually a definition
of the map (\cite[Prop. 4.7]{benoistRandomWalksReductive2016}).
Notice the reversed order of composition
in the properties above.

Dually,
define the
probability $ \mu_{\rho}^{t} $
as
the image of $ \mu_{\rho} $ by the adjoint map
$ g \mapsto g^{t} \in G^{*} = SL(V^{*}) $.
This defines a random walk on $ G^{*} $
which is also strongly irreducible and proximal,
and we call it the adjoint random walk.
There exists a unique
$ \mu_{\rho}^{t} $-stationary
measure $ \nu_{\rho}^{t} $ on $ \P V^{*} $
for the adjoint random walk,
and a limit map
$ \theta^{t} : \g{incr} \to \P V^{*} $
which satisfies:
for $ \bar{\mu} $ almost every
$ \xi = (\gamma_{n}) \in \g{incr} $,
any endomorphism $ \pi $
which is a limit point
of a sequence
$ a_{n} \rho(\gamma_{n} \dots \gamma_{1}) $ 
with $ a_{n} \in \C $
is of rank 1 and satisfies
$ \ker (\pi) = \theta^{t} (\xi) $
(again, notice the order of composition).
The adjoint limit map allows us to describe
the growth rate of individual vector:
\begin{propo}
	[\cite{benoistRandomWalksReductive2016}, Th. 4.8]
	For $ \bar{\mu} $-almost every
	$ \xi=(\gamma_{n}) \in \g{incr} $,
	for all $ v \in \P V $
	such that $ v \notin \theta^{t} (\xi) $
	we have
	\begin{equation}
		\lim
		_{n \to \infty}
		\frac{1}{n}
		\log
		\norm{
			\rho(
			\gamma_{n}
			\cdots
			\gamma_{1}
			)
			v
		}
		=
		\chi_{1} (\rho)
		.
	\end{equation}
	It implies that for every $ v \in V $
	and for $ \bar{\mu} $-almost every
	$ \xi \in \P $, the limit above holds.
\end{propo}
We see that the maximum growth rate
is attained by almost every path
of the random process.
This result can be used to prove the
Furstenberg's formula:
(\cite[Th. 48]{benoistRandomWalksReductive2016}):
\begin{equation}
	\int_{\Gamma}
	\int_{\P V}
	\log
	\frac{
		\norm{
			\rho(\gamma)v
		}}{
		\norm{v}
	}
	d \mu (\gamma)
	d \nu_{\rho} (v)
	=
	\chi_{1} (\rho)
	.
\end{equation}

We will need the
following
result which follows from
the average contraction
of the random walk
(quantified by $ \chi_{2} - \chi_{1} < 0 $).
Let $ P_{\rho} $ be the transition operator
defined for a continuous function $ f $
on $ \P V $ 
by:
\begin{equation}
	P_{\rho} f(x)
	=
	\int
	f(\rho(\gamma)x)
	d \mu (\gamma)
	.
\end{equation}
By results of \cite{lepageTheoremesLimitesPour1980},
there exists $ C >0 $,
$ \alpha \in (0,1) $
and $ \beta \in (0,1) $
such that
we have
\begin{equation}
	\norm{
		P_{\rho} ^{n} f
		-
		\int f d \nu_{\rho} 
	}
	_{C^{\alpha} } 
	\le
	C
	\beta^{n} 
	\norm{f}
	_{C^{\alpha} } 
	,
\end{equation}
The norm is the standard
Holder $ C^{\alpha} $
norm.
This result is called
the \emph{exponential contraction of the transition operator}.
It is a deep result in the theory which
implies many limit laws for the random walk.
It assumes an exponential moment for the random walk.

\subsubsection{Dual random walk}
\label{ssub:dual_random_walk}

A important theme in the present work
is the interplay between
the random walk and the dual
random walk.

Given, as in the previous paragraphs,
a probability 
$
\mu
$
on
$
\Gamma
$
and 
$
\rho
$
a linear representation
$
\Gamma \to G
$
the \emph{dual random walk}
is
the random walk on
$
G^{*} = SL(V^{*} )
$
induced
by the probability
$
\mu
$
and the representation 
$
\rho^{*} (\gamma)
=
\rho(\gamma^{-1})^{t} 
$
(remark that it is different
from the adjoint random walk
defined earlier).
The representation 
$
\rho^{*} 
$
is unbounded/irreducible/strongly irreducible/proximal if 
$
\rho
$
is.

The Lyapunov exponent
$
(
\chi^{*}_{i} 
)
$
of the dual random walk
satisfy
\begin{equation}
	\chi^{*}_{i} 
	=
	-
	\chi_{d + 2 - i} 
	.
\end{equation}

\subsection{Complex analysis}
\label{sub:reminder_of_complex_analysis}

We record here the standard notions
of complex analysis needed in this article.
A reference for all the facts
stated is \cite{demaillyComplexAnalyticDifferential}.
We fix a connected
complex manifold $ \Lambda $
of dimension $ N $.

A function
$ f : \Lambda \to
\R \cup \left\{ -\infty \right\} $
is \emph{plurisubharmonic}
if it is upper semicontinous
and if $ f \circ L $ is subharmonic
for every
holomorphic map
$ L: \D \to \Lambda $.
It is \emph{plurisurharmonic}
if $ -f $ is plurisubharmonic,
and \emph{pluriharmonic}
if it is plurisubharmonic and
plurisurharmonic.

Recall that a current $ T $
of degre $ p $
(and dimension $ 2N - p $)
on $ \Lambda $
is a continuous linear
form on the space
of differential forms of
degree $ N - p $
with compact support
(it is locally a differential
form of degree $ p $
with distribution coefficients).
The usual operations on forms
extend by duality on currents.
A current is of bidegree
$ (p,q) $
if it acts trivially
on forms of bidegree
different than
$ (N-p, N-q) $.
A current of bidegree 
$ (p,p) $
is \emph{positive}
if its local coefficients
are positive measure.
The main examples of positive current
are given by integration over an
analytic set:
by a theorem of Lelong,
if $ A $ is an analytic
set of pure dimension $ p $
we can define
the current of integration
over $ A $ by
\begin{equation}
	[A] \cdot \alpha
	=
	\int
	_{A_{reg} }
	\alpha
	,
\end{equation}
for $ \alpha $
a form of bidegree $ (p,p) $
with compact support,
where $ A_{reg} $
is the regular part of $ A $;
this current is well-defined
and is a positive current.

We denote by $ d $ the exterior derivative
on the differential forms on $ \Lambda $.
The complex structure induces
the decomposition $ d = d' + d'' $
where $ d' $ increases the bidegree
by $ (1,0) $ and $ d'' $ by $ (0,1) $.
We will use the real operator
$ d^{c} = \frac{1}{2 i \pi}(d' - d'') $.
It satisfies
$ dd^{c} = \frac{1}{i \pi} d' d'' $.
We have the following characterization
of plurisubharmoniticity:
an uppersemicontinuous function
$ f $
(which is not identically
$ -\infty $)
is plurisubharmonic if
and only if
it is locally $ L^{1} $
and
$ dd^{c} f $ is a positive
current
(it is always of bidegree $ (1,1) $).
It implies that a function
is pluriharmonic if and only if
$ dd^{c} f = 0 $ in the sense
of currents.

An important class of examples
of plurisubharmonic functions
is given by the Poincaré-Lelong
formula:
if $ f $ is a holomorphic
function which is not identically
zero
then $ \log \abs{f} $
is plurisubharmonic,
pluriharmonic outside
$ f^{-1} (0) $
and
\begin{equation}
	dd^{c} 
	\log \abs{f}
	=
	\sum
	m_{i} [Z_{i}]
	,
\end{equation}
where the $ Z_{i} $ 
are the irreducible components
of $ f^{-1} (0) $ and
the $ m_{i} $ are their
multiplicities.

Plurisubharmonic functions
are convenient to work with
because of their compactness
properties.
A version of
Hartog's Lemma
(see \cite{hormanderAnalysisLinearPartial2005})
states that
if $ (f_{n}) $ is a sequence
of plurisubharmonic functions
locally bounded above,
then either $ (f_{n}) $
converges uniformly to
$ -\infty $
or we can extract a sequence
$ (f_{n_{k} }) $ 
converging in $ L^{1}_{loc} $
to a plurisubharmonic function $ f $,
and  we have
$ \limsup f_{n_{k}}(x) = f(x) $
almost everywhere.
In particular,
if $ (f_{n}) $ is a sequence of
plurisubharmonic functions locally bounded
above and converging pointwise almost everywhere
to a function $ f $, then $ f $ is plurisubharmonic.
In the same way,
if $ (f_{n}) $ is a sequence of
\emph{pluriharmonic} functions locally bounded
and converging pointwise almost everywhere
to a function $ f $, then $ f $ is \emph{pluriharmonic}.

Plurisubharmonic functions are
subharmonic and therefore
satisfy the \emph{maximum principle}:
if a plurisubharmonic function
reaches its maximum
on an open set, it is
constant on this set.

\subsection{Family of representations}
\label{sub:family_of_representations}

We fix a connected 
complex manifold $ \Lambda $,
the parameter space.
Let
$ (\rho_{\lambda})
_{\lambda \in \Lambda}   $
be a holomorphic family
of representations
$ \Gamma \to G $.
It means that for all $\gamma$,
the map
$ \lambda \mapsto \rho_{\lambda} (\gamma)$ 
is holomorphic.

We consider the Lyapunov
exponents as functions on $ \Lambda $:
\begin{equation}
	\chi_{i} :
	\lambda
	\mapsto
	\chi_{i} ( \rho ( \lambda ) )
	.
\end{equation}
The functions
$ \chi_{1} + \dots + \chi_{i} $
are plurisubharmonic.
Indeed, we can use the norm
$ \norm{g}_{2}
=
\sqrt{ \sum g_{ij}^{2}}$
to define $ \chi $ and it is clear
that
\begin{equation}
	\frac{1}{n}
	\int_{\Gamma}
		\log
		\norm{ \wedge^{i} \rho ( \gamma ) }_{2}
		d \mu^{n} (\gamma)
\end{equation}
is a sequence of plurisubharmonic maps,
bounded above by the assumption of first moment,
converging pointwise to $ \chi_{1} + \dots + \chi_{i} $.

If at a parameter $ \lambda $ 
the representation $ \rho_{\lambda} $
is strongly irreducible and proximal,
we denote by
$ \nu_{\lambda} $
and $ \nu_{\lambda}^{*}  $ 
the associated stationary
measures
and by
$ \theta_{\lambda} $
and $ \theta_{\lambda}^{*} $
the associated limit maps.
We also denote by $ L(\lambda) $
its limit set.

We introduce a notion
of dynamical stability
for families of representations,
which is the main topic of this paper.
\begin{definition}
\label{def:prox_stable}
	The family of representations
	$ (\rho_{\lambda}) $
	is
	\emph{proximally stable}
	if for any $ \lambda_{0} \in \Lambda $
	and $ \gamma \in \Gamma $,
	if $ \rho_{\lambda_{0} } (\gamma) $
	is proximal
	then 
	$ \rho_{\lambda } (\gamma) $
	is proximal for all $ \lambda \in \Lambda $.
\end{definition}
In this case,
the fact that
$ \rho_{\lambda} (\gamma) $
is proximal is independent
of $ \lambda $,
and if it is the case
we say that $ \gamma $
is proximal.
If $ \gamma $ is proximal
it is clear that
$ \lambda \mapsto Fix^{+} \rho_{\lambda} (\gamma) $ 
and
$ \lambda \mapsto Fix^{-} \rho_{\lambda} (\gamma) $ 
are holomorphic,
and that the
logarithm of the
spectral radius
of $ \rho_{\lambda} (\gamma) $
is a pluriharmonic function
of $ \lambda $.
When the context is clear, we denote
these maps
(or their graph) by
$ Fix^{+} \gamma $
and 
$ Fix^{-} \gamma $.

We call
\emph{stability locus}
the largest 
subset of $ \Lambda $
on which $ (\rho_{\lambda}) $
is proximally stable.
Its complement is called
the \emph{bifurcation locus}.

We define currents
on $ \Lambda $ by
$ T_{i} = dd^{c} \chi_{i} $,
for $ i=1, \dots, d+1 $.
The complement of the
support of $ T_{i} $
is the largest open set
on which $ \chi_{i} $
is pluriharmonic.
Finally we set
$ T_{bif} = T_{1} + T_{d+1} $.

\section{Results}
\label{sec:statement}

From now on
we will assume the standing assumptions:
\begin{center}
	\emph{for all $ \lambda \in \Lambda $,
	      the representation
	      $ \rho_{\lambda} $ is
	      strongly irreducible
	      and proximal}
\end{center}
Our main result is:
\begin{theorem}
\label{th:main}
	Under the standing
	assumption,
	the Lyapunov exponents
	$ \chi_{1} $ and $ \chi_{d+1} $
	are pluriharmonic
	on $ \Lambda $
	if and only if
	$ (\rho_{\lambda}) $ 
	is proximally stable.
\end{theorem}


When the probability
$ \mu $ is symmetric
(that is
$ \mu ( \gamma )
= \mu ( \gamma^{-1} )$ 
for all $ \gamma \in \Gamma $)
the Lyapunov spectrum is symmetric:
$ \chi_{i} = - \chi_{d+2-i} $.
In this case $ \chi_{1} $
is pluriharmonic if and only if $ \chi_{d+1} $
is, so we can state:
\begin{corollary}
	If $ \mu $ is symmetric,
	under the standing assumption,
	$ \chi_{1} $ is pluriharmonic
	if and only if
	$ (\rho_{\lambda}) $ 
	is proximally stable.
\end{corollary}

\begin{Remark}
	We can state this theorem
	in term of the bifurcation
	current $ T_{bif} $:
	\begin{center}
		\emph{Under the standing
			assumptions,
			$ T_{bif} $
			vanishes
			on $ \Lambda $
			if and only if
			$ (\rho_{\lambda}) $ 
		is proximally stable,}
	\end{center}
	or in term of the bifurcation
	locus
	\begin{center}
		\emph{Under the standing assumptions,
			the support of $ T_{bif} $ is precisely
		the bifurcation locus.}
	\end{center}
\end{Remark}

We then show that
for proximally stable families
we can propagate some interesting
properties which are true at one
parameter, to every parameter.

\begin{theorem}
	\label{th:progation}
	Under the standing assumption,
	suppose that $ (\rho_{\lambda}) $
	is proximally stable.
	Let $ P_{\lambda}  $
	be one of the following
	property:
	\begin{enumerate}
		\item the representation
			$ \rho_{\lambda} $ 
			is faithful,
		\item the representation
			$ \rho_{\lambda} $ 
			is discrete,
		\item the representation
			$ \rho_{\lambda} $ 
			is projectively Anosov.
	\end{enumerate}
	If $ P_{\lambda_{0} } $ holds
	for some $ \lambda_{0} \in \Lambda $
	then $ P_{\lambda} $ 
	for every $ \lambda \in \Lambda $.
\end{theorem}
See section \ref{sub:the_anosov_property}
for a definition of the Anosov property.

\subsection{Hypothesis of the theorem}
\label{sub:hypothesis_of_the_theorem}

Let us explain the hypothesis of theorem
\ref{th:main}.
We assume that the representations
of the family are proximal,
otherwise the property of proximal stability
is trivially satisfied.
The assumption of strong irreducibility
is used to apply results from the
theory of random walks. 
In the section \ref{sub:reduction}
we explain how we can weaken
these asumptions,
to prove the "if" part of the main theorem.
More precisely
we can reduce the standing assumptions
of the direct implication 
of theorem \ref{th:main} to:
\begin{center}
	\emph{there exists one parameter $ \lambda_{0} $
		such that $ \rho_{\lambda_{0}}  $ 
		is strongly irreducible and proximal.
	}
\end{center}

\subsection{Generalization}
\label{sub:generalization}

Theorem \ref{th:main}
is a result about proximal actions
of
subgroups of
$ G = SL(V) $ acting on $ \P V $.
It is probable that it can be
extented to subgroups of a general semisimple
complex Lie group acting 
proximaly on
a flag variety.

In this work,
the random walk is used
as an accessory tool
to define the Lyapunov exponents
and to see the group
$ \Gamma $ as a dynamical system.
It is probable that we could
use a different dynamical system
and obtain the same result.
For example, when $ \Gamma $
is the fundamental group
of a compact surface,
we could use the geodesic flow
or the Brownian motion of
an hyperbolic structure
on this surface
as in \cite{deroinLyapunovExponentsSurface2015}.
A more challenging problem would
be to replace the random walk
by a subshift of finite type
(the random walk corresponding
to the full shift),
see \cite{kasselEigenvalueGapsHyperbolic2020}
where this direction is explored
for Anosov representations.


\subsection{Equidistribution}
\label{sub:equidistribution}

The theorem \ref{th:main}
states that the bifurcation locus
coincides with the support
of 
the bifurcation current
$ T_{bif} $.
The interest of this result
is that we can now use tool
from potential theory to describe
the bifurcation locus.
The general machinery
developed in
\cite{deroinRandomWalksKleinian2012}
applies in our setting (with identical proofs)
and we obtain equidistribution results.
More precisely:

Let $ (\rho_{\lambda})_{\lambda \in \Lambda}  $
be a family of representation.
Fix $ t \in C $ 
and
let $ Z(\gamma, t) $ be
the subset of $ \Lambda $
composed of $ \lambda $
such that
$ \tr (\rho_{\lambda}(\gamma)) = t $.
\begin{propo}
	Let $ (\rho_{\lambda}) $
	be a family of representations
	satifsying the standing assumptions.
	Suppose that $ \mu $ is symmetric.
	Then for $ \bar{\mu} $-almost every
	$ (\gamma_{n}) \in \g{incr} $,
	the sequence of currents of integration
	\begin{equation}
		\frac{1}{n}
		\left[ Z(\gamma_{n} \cdots \gamma_{1}, t) \right]
	\end{equation} 
	converges to
	$ T_{bif} $.
\end{propo}

\subsection{Outline of the proof}
\label{sec:outline}

The easy implication (proximally stable
implies pluriharmonicity of the Lyapunov exponent)
is proved in section \ref{sec:stab_prox_PH}.

For the reverse implication,
the first step is to prove
the existence of an
\emph{holomorphic variation of
the Poisson boundary},
an object which binds
in a family of holomorphic map
all the boundary map of the
family of representation
(section \ref{sub:hol_poisson}).
To construct this variation,
we need a result which allows us
to control the volume growth
of the trajectories of the random
walk applied to certain graphs,
in term of the bifurcation current.
(section \ref{sub:vol_bif}).
Then we show that the holomorphic
variation of the Poisson boudary
is compact, meaning that
the associated family of map
is relatively compact in the uniform topology
(section \ref{ssub:compactness}).
To prove the compactness,
we show a non-intersection
property between the graphs
of the variation of the Poisson
boundary and the graphs
of the dual variation of the Poisson
boundary (section \ref{sub:intersections}).
Using the compactness,
we can conclude and prove
that if a matrix is proximal
at some parameter, it stays
proximal for every parameters
(section \ref{sub:fixed_points_hol}).

In the last section \ref{sec:stab_consequences}
we prove some consequences of the
proximal stability.

\subsection{Comparison of the different notions of stability}
\label{sub:comp}

In this section we
elaborate on the comparison
of the different notions of
stability we hinted to in the introduction,
and compare them with the one
we introduce in this article.

For endomorphisms of 
$
\P^{1} (\C)
$
Mañé, Sad and Sullivan
\cite{maneDynamicsRationalMaps1983}
and independently
Lyubich
\cite{lyubichTypicalPropertiesDynamics1983}
introduced a notion of stability
for holomorphic family.
They showed that
the stability of
a family 
$
(f_{\lambda})
_{\lambda \in \Lambda} 
$
can be defined by
the following equivalent properties,
among others
(for connected $\Lambda$):

\begin{enumerate}
	\item the periodic points
		of 
		$
		(f_{\lambda})
		$
		do not change type.
	\item the Julia sets
		$
		(J_{\lambda})
		$
		vary continuously
		in the Hausdorff topology
	\item all the 
	$
	(f_{\lambda})
	$
	are quasiconformally-conjugated
	on their Julia set.
\end{enumerate}
The decomposition stability-bifurcation
of the parameter space
is analogous to the decomposition
Julia-Fatou for a single endomorphism.

For representations
with value in $ \text{SL}(2, \C) $,
Sullivan
\cite{sullivanQuasiconformalHomeomorphismsDynamics1985}
proved that very different
notions of stability are in fact
equivalent.
Let $ \rho: \Gamma \to SL(2, \C) $ 
be a
faithful,
unbounded, non-rigid,
torsion-free
representation
with $ \Gamma $
a
finitely generated group.
Then the following are equivalent:
\begin{enumerate}
	\item $ \rho $ is convex-cocompact.
	\item $ \rho $ is algebraically stable.
	\item $ \rho $ is structurally stable
	\item The type of the elements of
		$ \rho $ is stable.
	\item $ \rho $ satisfies an
		hyperbolicity condition
		on its limit set.
\end{enumerate}
Convex-cocompact is a geometric property.
Algebraically stable means
that nearby representations stay
faithful; it is an algebraic property.
Structurally stable means
that the action
on $ \P^1 $
of nearby representations
are (smoothly) conjugated to
the one of
$ \rho $ on their limit set:
it is a dynamical property.
The stability of the type of elements of
$ \rho $ means that the type
(hyperbolic, parabolic or elliptic)
of a matrix $ \rho(\gamma) $ doesn't
change for nearby deformation of $ \rho $.
The hyperbolicity condition is a property
of expansion, analogue to the one
found in hyperbolic dynamics.
It is a remarkable fact,
specific to the rank 1,
that such different properties
are equivalent.
Observe that for representations
as well as for endomorphisms,
the $\lambda$-lemma
is a crucial tool to construct
conjugacies for stable families.

When generalizing to higher rank
groups, the equivalences
breakdown, a priori
(for instance, in the
case of representations
of
$ \Z $,
one easily sees that
algebraic stability
is not equivalent
to structural stability).
As said before,
the right generalization
of convex-cocompact in higher
rank is the notion of Anosov representation.
It is known that they enjoy
all of the previously defined
properties of stability.
The situation is less clear
for the reciprocal.
In \cite{kapovichSullivanStructuralStability2019},
Kapovich, Kim and Lee
explore the generalization
of the hyperbolicity property
and show that being Anosov
is equivalent to a kind
of hyperbolicity property
(see \cite[th. 6.3]{kapovichSullivanStructuralStability2019}).

In this article we study
the notion of proximal stability
(for family of representations)
(see the definition \ref{def:prox_stable} in section 
\ref{sec:statement}).
It is analoguous to the
notion of stability of
repelling $ J $-cycles
in \cite{bertelootDynamicalStabilityLyapunov2018}
for families of endomorphisms
of $ \P^{d} (\C) $ 
(which requires that
repelling cycles move holomorphically)
and
the notion of
weak
$ J^{*} $-stability
of
\cite{dujardinStabilityBifurcationsDissipative2014}
for polynomial automorphism
of $ \C^{2} $ 
(which requires that
saddle periodic points move
holomorphically).
We think of it as the generalization
of the \say{stability of type of elements}
property for $ SL(2, \C) $.
We say that a representation
is proximally stable
if the family composed of
all nearby deformations
is proximally stable.
It is clear that Anosov representations
are proximaly stable.
It would be an interesting question
to study if
proximally stable representations
are
also
Anosov.
For now, we can prove
that proximaly stables representations
statisfy some kinds of stability:
they are algebraically stable,
stably discrete,
and stably Anosov,
see section \ref{sec:stab_consequences}.
The propagation of the Anosov property
for proximally stable families of representations
is an analogue of the propagation
of uniform hyperbolicity
for weakly stable families of polynomial
automorphisms in
\cite{bergerStabilityHyperbolicityPolynomial2017}.

\section{Proximally stable implies pluriharmonicity of the Lyapunov exponent}
\label{sec:stab_prox_PH}

In this section we prove the \say{easy} direction of the
main theorem \ref{th:main}.
Let $ ( \rho_{\lambda} ) $
be a family of representations satisfying the standing
assumptions and which is proximally stable.
We show that $ \chi_{1} $ and $ \chi_{d+1} $
are pluriharmonic.


By the proposition 4.2 in \cite{deroinRandomWalksKleinian2012},
for almost every $ (\gamma_{n}) \in \g{incr} $ we have:
\begin{equation}
\label{eq:lyap_conv}
	\frac{1}{n}
	\log
	\norm{
		\rho_{\lambda}(
		\gamma^{(n)}
		)
	}
	\to
	\chi_1(\lambda)
	,
\end{equation}
where the convergence holds in $ L^{1}_{loc} $ .
Remark here that the $ (\gamma_{n}) $ can be
choosen independently of $ \lambda $.
We now show that it holds also
when we replace the norm
$ \norm{\cdot} $
by the spectral radius $ r $.
For that we use the following
lemma, which is a variation of 
the Hartog's Lemma:

\begin{lemma}[\cite{deroinRandomWalksKleinian2012}, Lemma 4.3]
\label{lem:compared_conv_psh}
	Let
	$ (u_{n}) $
	be a sequence of plurisubharmonic
	functions converging in $ L^{1}_{loc}$ 
	to a continuous plurisubharmonic
	function $ u $
	and
	$ (v_{n}) $
	another sequence of plurisubharmonic
	functions such that
	\begin{itemize}
		\item
			$ v_{n} \le u_{n} $
			pointwise,
		\item
			$ v_{n} (x) \to u_{n} (x) $ 
			for $ x $ in a dense subset,
	\end{itemize}
	then $ v_{n} \to u $
	in $ L^{1}_{loc} $.

\end{lemma}

Let $ (\lambda_{k}) $
be a dense sequence of parameters
in $ \Lambda $.
For almost all $ (\gamma_{n}) $
we have that:
\begin{align}
	\frac{1}{n}
	\log
	r(
	\rho_{\lambda_{k} }
	(\gamma^{(n)} )
	)
	& \to
	\chi_{1}  (\lambda_{k} )
	,\\
	\frac{1}{n}
	\log
	\norm{
		\rho_{\lambda}(
		\gamma^{(n)}
		)
	}
	&\to
	\chi_{1} (\lambda)
	,
\end{align}
where the second
convergence is in $ L^{1}_{loc} $
and that $ \rho(\gamma_{n}) $
is proximal for $ n $
large enough
(see \cite[Th. 4.12]{benoistRandomWalksReductive2016}).
Fix a sequence $ (\gamma_{n}) $
such that \ref{eq:lyap_conv}
and these hold.

We define:
\begin{equation}
	u_{n} :
	\lambda
	\mapsto
	\frac{1}{n}
	\log
	\norm{
		\rho_{\lambda}(
		\gamma^{(n)}
		)
	}
	,
\end{equation}
and
\begin{equation}
	v_{n} :
	\lambda
	\mapsto
	\frac{1}{n}
	\log
	r(
	\rho_{\lambda}
	(\gamma^{(n)} )
	)
	.
\end{equation}
All of these functions are plurisubharmonic
and
we have $ v_{n} \le u_{n} $,
$ u_{n} \to \chi_{1}  $ and
$ v_{n} (\lambda_{k}) \to \chi_1 (\lambda_{k}) $.
By the Lemma \ref{lem:compared_conv_psh},
we have that $ u_{n} $ converges to $ \chi_{1}  $ in
$ L^{1}_{loc} $.

Moreover, each $ v_{n} $ is pluriharmonic
thanks to the proximal stability.
This sequence is bounded below by $ 0 $
and bounded above by the moment assumption.
It follows, by Hartog's lemma,
that the limit $ \chi_{1} $ is also
pluriharmonic.

We now deduce the pluriharmonicity of $ \chi_{d+1} $ 
from the pluriharmonicity of $ \chi_{1} $:

The top Lyapunov exponent
$ \chi_{1}^{*}  $
of the dual representation is pluriharmonic by the
argument above.
We have that
$
\chi_{1} ^{*} 
=
- \chi_{d+1}
$,
so $ \chi_{d+1}$
is also pluriharmonic.

\section{Pluriharmonic Lyapunov exponent implies proximal stability}

\subsection{A series of reductions}
\label{sub:reduction}

In this section we justify
the claim made in the discussion
of the hypothesis \ref{sub:hypothesis_of_the_theorem},
that we can weaken the standing assumption to
\begin{center}
	\emph{there exists one parameter $ \lambda_{0} $
		such that $ \rho_{\lambda_{0}}  $ 
		is strongly irreducible and proximal.
	}
\end{center}
We also explain why we can
assume in the proof of the theorem
that the dimension of $ \Lambda $
is 1.

\subsubsection{Reduction to the one dimensional case}
\label{ssub:reduction_one_dim}

Assume the theorem \ref{th:main} holds
if $ \Lambda $ is a complex manifold
of dimension 1.
Let $ (\rho_{\lambda})_{\lambda \in \Lambda} $
be a family of representations satisfying
the standing assumptions and that $ \chi_{1} $
and $ \chi_{d+1} $ are pluriharmonic.
Let $ \lambda_{0},\lambda_{1} \in \Lambda $
and suppose that $ \rho_{\lambda_{0} } (\gamma) $
is proximal for some $ \gamma \in \Gamma $.
We need to show that $ \gamma $ is proximal
at $ \lambda_{1} $.
It is enough to show that $ \lambda_{0} $
and $ \lambda_{1} $ can be connected by
a sequence of intersecting one dimensional
complex manifolds.
By compactness of a path
joining $ \lambda_{0} $
to $ \lambda_{1} $ 
we can find a sequence
$ B_{0}, \dots, B_{N} $
of open sets biholomorphic to a standard ball
such that $ \lambda_{0} \in B_{0} $,
$ \lambda_{1} \in B_{N} $ and
$ B_{i} \cap B_{i+1} \neq \emptyset $.
For each $ i $ we can find
$ D_{i} \subset B_{i} $
such that $ D_{i} $ is a submanifold
biholomorphic to a 
one-dimensional complex disk,
$ \lambda_{0} \in D_{0} $,
$ \lambda_{1} \in D_{N} $ and
$ D_{i} \cap D_{i+1} \neq \emptyset $.
By applying successively the one-dimensional
case theorem to the $ \rho_{|D_{i}} $,
if follows that $ \gamma $ is proximal
on each $ D_{i} $ and then at $ \lambda_{1} $.


Now that we assume $ dim \Lambda = 1 $,
the notions of
pluriharmonicity, plurisubharmonicity, etc.
are equivalent to
harmonicity, subharmonicity, etc.,
so we will drop the prefix "pluri" from now on.




\subsubsection{Proximality at every parameters}
\label{ssub:prox_all}

We show that if
$ (\rho_{\lambda}) $ is a family
which is strongly
irreducible at every parameters
and proximal at some parameter
then it is proximal at every parameter,
when $ \chi_{1} $ is harmonic.

By \cite[Cor. 10.5]{benoistRandomWalksReductive2016},
the representation $ \rho_{\lambda} $ is proximal
if and only if
$ \chi_{1} (\lambda)
>
\chi_{2} (\lambda)
$.
By assumption this inequality
is true at some parameter.
The function
$ \chi_{1} + \chi_{2} $
is subharmonic
and $ \chi_{1} $ is harmonic,
so $ \chi_{2} $ is subharmonic.
The function $ \chi_{1} - \chi_{2} $
is then superharmonic and non-negative.
By the minimum principle,
if it vanishes at some point
it's constantly equal to $ 0 $.
It is not the case,
so the function is positive.
We conclude that for all $ \lambda $,
$ \chi_{1} (\lambda)
> \chi_{2} (\lambda)$
and that $ \rho_{\lambda} $ is proximal.

\subsubsection{Irreducibility and strong irreducibility}
\label{ssub:irr_stron_irr}

We show that if $ (\rho_{\lambda}) $ is irreducible
at every parameters and strongly irreducible
and proximal at at least one parameter,
then it is strongly irreducible at every parameters
if $ \chi_{1} $ is harmonic.

By the reasoning of the previous section \ref{ssub:prox_all},
we have $ \chi_{1} > \chi_{2} $ everywhere.
It implies that the family is proximal
at every parameter, by irreducibility.
It is enough to show the following
(classical) lemma for a fixed
representation:
\begin{lemma}
	If
	$ \rho : \Gamma \to G $
	is an irreducible representation 
	which is proximal
	then
	$ \rho $ is strongly irreducible.
\end{lemma}
\begin{proof}
	First, remark that there exists
	a proximal element in the image of
	$ \rho $, because $ \rho $ is proximal
	and irreducible.

	Now we assert that an irreducible representation
	with a proximal element is strongly irreducible.
	Consider a virtually invariant, strongly
	irreducible subspace $ W $:
	it's a subspace such that $ \rho(\Gamma) W $ is finite,
	and no proper subspace is virtually invariant.
	Assume that $ W $ is not all of $ \C^{d+1} $,
	that is $ \rho $ is not strongly irreducible.
	Let $ L^{+} $ and $ H^{-} $ be the attractive line
	and repulsive hyperplane of $ g $.
	Observe that the orbit $ \rho(\Gamma) W $
	spans a vector space which is stable by
	$ \rho $: by irreducibility it must be the whole
	space. In particular, 
	the orbit $ \rho(\Gamma) W $ 
	is not
	contained in $ H^{-}$.
	Remark that for distincts $ W_{1}, W_{2} \in \rho(\Gamma)W $,
	the intersection $ W_{1} \cap W_{2} $ is trivial.
	Otherwise, it would also be a strongly irreducible virtually
	invariant subspace and a proper subspace of $ W_{1} $.

	The line $ L^{+} $ is in some $ W_{1} \in \rho(\Gamma)W $
	because $ \bigcup \rho(\Gamma)W $ is closed,
	stable by $ g $
	and if
	$ z \in \bigcup \rho(\Gamma)W \setminus H^{-} $ 
	then $ g^{n} z $ converges to $ L^{+} $.
	Then $ W_{1} $ is stable by $ g $, because
	$ g W_{1} \cap W_{1} $ contains $ L^{+} $.
	Now, considering the action of $ g $ on the
	finite set $ \rho (\Gamma) W $,
	we see that the reunion
	$ \bigcup \rho(\Gamma) W \setminus W_{1}  $
	is stable by $ g $.
	But for any
	$ z \in
	\bigcup \rho(\Gamma) W \setminus W_{1} $,
	the iterates $ g^{n} z $ converges to $ L^{+} $,
	which is absurd because
	$ L^{+} $ is at a positive (projective) distance
	from
	$\bigcup \rho(\Gamma) W \setminus W_{1} $.
\end{proof}

\subsubsection{Irreducibility}
\label{ssub:red_irr}

We first show that if $ (\rho_{\lambda})_{\lambda \in \Lambda}  $
is a family
which is strongly irreducible and proximal at
one parameter $ \lambda_{0} $,
then it is strongly irreducible and proximal
on a Zariski dense set $ \Lambda_{0} \subset \Lambda $,
if $ \chi_{1} $ is harmonic.
By the results of sections \ref{ssub:prox_all}
and \ref{ssub:irr_stron_irr},
it is enough to show that
the family is irreducible on a Zariski dense set.
This is true, because being irreducible
is an Zariski open condition
(see for example \cite[prop. 27]{sikoraCharacterVarieties2012}).

Now, if we assume that the family
$ (\rho_{\lambda}) $
is strongly irreducible and proximal at
one parameter $ \lambda_{0} $,
then it is strongly irreducible and proximal
on a Zariski dense set $ \Lambda_{0} \subset \Lambda $,
and we can apply the theorem \ref{th:main}
to the family $ (\rho_{\lambda})_{\lambda \in \Lambda_{0} } $.
To get the stability of proximal elements
at the discret set of points $
\Lambda \setminus \Lambda_{0} $,
we apply the following lemma \ref{lem:prox_zero}
at these points.

\begin{lemma}
\label{lem:prox_zero}
	Let $ M(\lambda) $
	be a holomorphic
	family of matrices
	indexed by
	$ \D $ such that
	for $ \lambda \neq 0 $
	the matrix $ M(\lambda) $
	is proximal.
	Then $ M(0) $ is proximal.
\end{lemma}
\begin{proof}
	Let $ \mu_{1}(\lambda) $
	and $ \mu_{2}(\lambda) $
	be the eigenvalues of
	greatest and second greatest
	modulus of $ M(\lambda) $,
	for $ \lambda \neq 0 $.
	Let
	$ l_{1} = \log \abs{\mu_{1}} $ 
	and
	$ l_{2} = \log \abs{\mu_{2}} $.
	The function $ l_{1} $ is
	harmonic, because $ \mu_{1} $
	is holomorphic.
	The function $ l_{2} $ is
	subharmonic and continuous
	because it is
	the logarithm of the spectral
	radius of the restriction
	of $ M(\lambda) $ to its repelling
	hyperplane.
	Let $ f = l_{1} - l_{2} $.
	It is a continous, superharmonic,
	positive function on $ \D^{*} $.
	It extends to a continous function
	at $ 0 $ (see \cite[II.2]{katoPerturbationTheoryLinear1995}).
	We show that such a function can't
	vanish at $ 0 $.
	Up to restricting and rescaling,
	we can suppose that $ f $ is
	defined and continuous on $ \partial \D $.

	We pullback $ f $ to
	the univeral covering of
	$ \D^{*} $:
	let
	$ F = f \circ \pi $
	where $ \pi : \HH \to D^{*} $
	is given by
	$ \pi(z) = \exp (iz) $.
	The function $ F $ satisfies
	the same properties as $ f $:
	it is superharmonic, continuous,
	and positive. We show that
	it doesn't vanish at infinity.

	Define, for $ x + iy \in \HH $,
	\begin{equation}
		P (x + iy)
		=
		\frac{y}
		{x^{2} + y^{2}}
		,
	\end{equation}
	the Poisson kernel for $ \HH $.
	As $ F $ is superharmonic we have
	for all $ z \in \HH $:
	\begin{equation}
		F(z)
		\ge
		\int_{-\infty}^{+\infty} 
		P(z + t)
		F(t)
		dt
		.
	\end{equation}
	The map $ F $ is continous, positive
	and periodic on the real line $ \R $,
	so it is bounded below by some $ m > 0 $.
	We have for $ y > 0 $:
	\begin{align}
		F(iy) &\ge 
		m
		\int_{-\infty}^{+\infty} 
		P(iy + t)
		dt \\
		& \ge
		m
		\int_{y}^{+\infty} 
		P(iy + t)
		dt \\
		& =
		m
		\int_{y}^{+\infty} 
		\frac{y}{y^2 + t^2 }
		dt \\
		& \ge
		m
		\int_{y}^{+\infty} 
		\frac{y}{2 t^2 }
		dt \\
		& =
		 \frac{m}{2}
		 > 0
		 ,
	\end{align}
	so $ F $ doesn't vanish at infinity,
	and $ f $ doesn't vanish at $ 0 $.
	It follows that $ M $ stays proximal at $ 0 $.
\end{proof}

\subsection{The volume of graphs is controlled by the bifurcation current}
\label{sub:vol_bif}

In this section we show
the Proposition
\ref{th:bounded_vol}
that asserts that
the volume of iterates of graphs
under a path of the random walk
is controlled, in average,
by the bifurcation current.
First we introduce some
notations
and show some technical lemmas.


The results of this section are
local in nature,
so we will assume that
$ \Lambda $ is the unit disk.
A family $ (\rho_\lambda)_{\Lambda} $
of representations
naturally gives an action
of $ \Gamma $ on objects fibered over $ \Lambda $.
For example, it acts on a function
$ f:\Lambda \to \P V $ by:
\begin{equation}
	\gamma \cdot f
	:=
	\rho (\gamma) f
	=
	\lambda
	\mapsto
	\rho_{\lambda} (\gamma)
	\cdot
	f(\lambda)
,
\end{equation}
and on function
$ g:\Lambda \to \P V^{*} $ by:
\begin{equation}
	\gamma \cdot g
	:=
	\rho^{*}  (\gamma) f
	=
	\lambda
	\mapsto
	\rho_{\lambda} (\gamma^{-1} )^{t} 
	\cdot
	f(\lambda)
.
\end{equation}

For a holomorphic map $ \sigma $ from
$ \Lambda $ to $ \P V $
we denote $ \Gamma(\sigma) $ its graph,
which is a complex submanifold
of $ \Lambda \times \P V $
of dimension $ 1 $.


We recall that for an analytic set $ A $
of codimension $ k $,
we denote by $ [A] $ the current of integration
over $ A $ which is a positive $ (k,k)$-current.

We denote by $ \pi_{1} $ and $ \pi_{2} $
the projections $ \Lambda \times \P V $
over $ \Lambda $ and $ \P V $.
We can suppose $ \Lambda $
is equiped with
its standard Kälher form,
$ \omega_{1} $
and we equip $ \P V $ with the
Fubini-Study metric $ \omega_{2} $.
We equip $ \Lambda \times \P V $
with the Kälher form
$ \omega =
\pi^{*}_{1} \omega_{1}
+
\pi^{*}_{2} \omega_{2}
$ 
.

For a holomorphic map
$ f : U \to \P V $,
we define
$ dd^{c} \abs{f} $
as the following.
Let $ F $ be a local
holomorphic
lift of $ f $
to $ V $.
Then
$ dd^{c} \log \norm{F} $
is (locally) well-defined
and another choice of
$ F $ gives the same result.
We denote by
$ dd^{c} \log \abs{f} $
the (global)
current on $ U $
it defines.
In fact we have
$ dd^{c} \log \abs{f}
= f^{*} \omega_{2} $
(see \cite[Ex. 6.2.14]{dinhIntroductionTheoryCurrents2005}).

If $ T $ is a positive $ (1,1) $-current
on $ \Lambda $, we denote by $ \norm{T}_{K} $
its mass over the (relatively) compact subset $ K $
defined by
\begin{equation}
	\norm{T}_{K} 
	:=
	\int_{K} T
	.
\end{equation}

The functions
$ \chi_{1} $
and $ \chi_{d+1} $
are plurisubharmonic
on $ \Lambda $.
We define
the positive $ (1,1) $-currents
$ T_{1} $,
$ T_{d+1} $ and $ T_{bif} $ on $ \Lambda $ by
\begin{align}
	T_{1}
	&=
	dd^{c} \chi_{1}
	,\\
	T_{d+1}
	&=
	dd^{c} \chi_{d+1} 
	\\
	T_{bif} &= T_{1} + T_{d+1} 
	.
\end{align}

The following lemma gives a
formula
for the volume of a graph.
Let $ U $ be a relatively
compact open set in $ \Lambda $.

\begin{lemma}
	\label{lem:volume_mass}
	Let $ f $ be a holomorphic map
	from $ U $ to $ \P V $. Then:
	\begin{equation}
		\Vol(\Gamma(f))
		=
		\Vol(U)
		+
		\norm{
		(dd^{c} \log \abs{f}) 
		}_{U} 
		.
	\end{equation}
\end{lemma}

\begin{proof}
	We have
	\begin{equation}
		\Vol( \Gamma_{f} )
		=
		\int_{\Gamma_{f}}
		\omega
		=
		\int_{\Gamma_{f}}
		\pi_{1}^{*} \omega_{1} 
		+
		\pi_{2}^{*} \omega_{2} 
		.
	\end{equation}
	The restriction of
	the projection
	$ \pi_{1}
	: \Gamma_{f} \to U $
	induces a biholomorphism
	(that we denote by the same symbol).
	It gives, for the first term of the sum:
	$	
	\int_{\Gamma_{f}}
	\pi_{1}^{*} \omega_{1} 
	=
	\int_{U}
	\omega_{1}
	=
	\Vol(U)
	$.

	For the second term of the sum,
	remark that
	$ \pi_{2}
	\pi_{1}^{-1} $
	is just $ f $.
	We then have:
	\begin{align}
		\int_{\Gamma_{f}}
		\pi_{2}^{*} \omega_{2}
		&=
		\int_{\Gamma_{f}}
		\pi_{1}^{*}
		\pi_{1}^{-1*}
		\pi_{2}^{*} \omega_{2}
		\\
		&=
		\int_{U} 
		(\pi_{2}\pi_{1}^{-1}) ^{*} \omega_{2}
		\\
		&=
		\int_{U} 
		f^{*} \omega_{2}
		.
	\end{align}
	Recall that
	$ f^{*} \omega_{2} = dd^{c} \log \abs{f} $,
	so finally:
	\begin{equation}
		\int_{\Gamma_{f}}
		\pi_{2}^{*} \omega_{2}
		=
		\int_{U} 
		dd^{c} \log \abs{f}
		=
		\norm{
			dd^{c} \log \abs{f}
		}_{U} 
		.
	\end{equation}
\end{proof}

To use the Lemma \ref{lem:volume_mass},
with $ f $ of the form
$ \lambda \mapsto \gamma^{(n)}_{\lambda} z_{0} $
we need estimates on the convergence
of potentials to the Lyapunov exponent.

\begin{lemma}
	\label{lem:speed_conv}
	Under the standing assumptions
	on the family of representations
	and $ \mu $,
	for any $ x \in \P V $
	we have
	\begin{equation}
		\chi^{(n)}_{1} (\lambda)
		:=
		\int
		\log
		\frac{
			\norm{
				\rho_{\lambda} 
				(\gamma^{(n)})
				x
			}
			}{ \norm{x}
		}
		d \mu^{n}(\gamma)
		=
		n \chi_{1} (\lambda)
		+
		O(1)
		,
	\end{equation}
	where the $ O(1) $ 
	is locally uniform in $ \lambda $ 
	in the $ L^{\infty} $
	norm, and:
	\begin{equation}
		dd^{c} \chi^{(n)}_{1} 
		=
		n dd^{c} \chi_{1} 
		+ O(1)
		,
	\end{equation}
	where the $ O(1) $
	is in the mass norm
	over any compact.
\end{lemma}
\begin{proof}
	We sketch the argument, which is the same as in
	\cite[Prop. 3.8]{deroinRandomWalksKleinian2012}.
	We define $ \varphi $ on $ \P V $ 
	by
	\begin{equation}
		\varphi(x)
		=
		\int
		\log
		\frac{
			\norm{\rho_{\lambda} (\gamma)x}
		}
		{
			\norm{x}
		}
		d \mu (\gamma)
		,
	\end{equation}
	and by Furstenberg's formula we have
	$
	\int \varphi d\nu_{\lambda} 
	=
	\chi_{\lambda} (\lambda)
	.
	$
Recall that $ P_{\lambda} $ is the transition operator.
	We compute:
	\begin{align}
		&
		\int
		\log
		\frac{
			\norm{
				\rho_{\lambda} (
				\gamma_{n} \cdots \gamma_{1}
				)
				x
			}
		}
		{
			\norm{x}
		}
		d\mu(\gamma_{n})
		\cdots
		d\mu(\gamma_{1})
		\\
		&=
		\sum_{k=0}^{n-1} 
		\int
		\log
		\frac{
			\norm{
				\rho_{\lambda} (
				\gamma_{k+1} \cdots \gamma_{1}
				)
				x
			}
		}
		{
			\norm{
				\rho_{\lambda} (
				\gamma_{k} \cdots \gamma_{1}
				)
				x
}
		}
		d\mu(\gamma_{k})
		\cdots
		d\mu(\gamma_{1})
		\\
		&=
		\sum_{k=0}^{n-1} 
		P_{\lambda} ^{k} \varphi (x)
		,
	\end{align}
	so,
	\begin{equation}
		\int
		\log
		\frac{
			\norm{
				\rho_{\lambda} (
				\gamma_{n} \cdots \gamma_{1}
				)
				x
			}
		}
		{
			\norm{x}
		}
		d\mu(\gamma_{n})
		\cdots
		d\mu(\gamma_{1})
		-
		n \chi_{1} (\lambda)
		=
		\sum_{k=0}^{n-1} 
		(
		P_{\lambda} ^{k} \varphi (x)
		-
		\int \varphi d \nu_{\lambda} 
		)
	\end{equation}
	and as we have
	\begin{equation}
		\norm{
			P_{\lambda} ^{n} \varphi
			-
			\int \varphi d \nu_{\lambda} 
		}
		_{C^{\alpha} } 
		\le
		C
		\beta^{n} 
		\norm{\varphi}
		_{C^{\alpha} } 
		,
	\end{equation}
	the last sum is bounded,
	independently of $ x $.
	It implies the estimate
	at a fixed representation:
	\begin{equation}
		\log
		\frac{
			\norm{
				\rho(
				\gamma_{n} \cdots \gamma_{1}
				)
				x
			}
		}
		{
			\norm{x}
		}
		d\mu(\gamma_{n})
		\cdots
		d\mu(\gamma_{1})
		=
		n \chi_{1} 
		+
		O(
		1
		)
		.
	\end{equation}
	The estimate is 
	locally uniform
	in $ \lambda $ 
	because the constants
	$ \alpha, C, \beta $
	can be chosen
	locally independently of
	$ \lambda $,
	because they depend
	only of the average
	contraction
	\begin{equation}
		\sup_{x \neq y \in \P V } 
		\int
		\left( 
			\frac{
				d(\rho_{\lambda} (\gamma)x,
				\rho_{\lambda} (\gamma)y)
				}{
				d(x,y)
			}
		\right)^{\alpha_{0} }
		d \mu^{*n_{0} } (\gamma)
		< 1
		,
	\end{equation}
	(for some $ \alpha_{0} $, $ n_{0} $)
	and the exponential moment condition
	on $ \mu $.

	The second estimate is deduced
	from the first by an application
	of the Chern-Levine-Nirenberg inequality
	(see \cite[prop. 3.3]{demaillyComplexAnalyticDifferential}).
\end{proof}

From these lemmas we can deduce
the main result of this section,
a first step for the study
of the stability of fixed points.

\begin{propo}
	\label{th:bounded_vol}
	Let $ \sigma_{0} $
	be a constant map
	$ U \to \P V $.
	Under the standing assumptions,
	if $ T_{1} =0 $ on $ U $,
	the mean volume of the graphs
	of maps
	$ \gamma^{(n)} \sigma_{0} $	
	is bounded,
	that is
	\begin{equation}
		\int
		\Vol(
		\Gamma( \gamma^{(n)} \sigma_{0} )
		)
		d \mu^{n} (\gamma)
		=
		O(1)
		.
	\end{equation}
\end{propo}
\begin{proof}
	By lemma \ref{lem:volume_mass},
	it is enough to bound:
	\begin{equation}
		\int
		\norm{
			dd^{c} \log \abs{\gamma^{(n)} \sigma_{0}}
		}_{U} 
		d \mu^{n} (\gamma)
		=
		\norm{
		\int
			dd^{c} \log \abs{\gamma^{(n)} \sigma_{0}}
		d \mu^{n} (\gamma)
		}_{U} 
		,
	\end{equation}
	A lift
	of $ \gamma^{(n)} \sigma_{0} $
	to $ V $ 
	is given by
	\begin{equation}
		F_{n} (\lambda)
		=
		\rho_{\lambda}(\gamma^{(n)})
		\cdot v_{0},
	\end{equation}
	where $ v_{0} $ is a non-zero
	vector in the line given by $ \sigma_{0} $.
	By lemma \ref{lem:speed_conv},
	\begin{equation}
		dd^{c} 
		\int
		\log \norm{F_{n}(\lambda)}
		d \mu^{n} (\gamma)
		= O(1)
		,
	\end{equation}
	in the mass norm,
	as we have
	$ dd^{c} \chi_{1} = T_{1} = 0 $.
	This means that
	\begin{equation}
		\norm{
		\int
			dd^{c} \log \abs{\gamma^{(n)} \sigma_{0}}
		d \mu^{n} (\gamma)
		}_{U} 
		= O(1)
		,
	\end{equation}
	and the claim follows.
\end{proof}

We can state this result
for the dual random walk:
\begin{propo}
	\label{th:bounded_vol_dual}
	Let $ \varphi_{0} $
	be a constant map
	$ U \to \P V^{*} $.
	Under the standing assumptions,
	if $ T_{d+1} =0 $ on $ U $,
	we have
	\begin{equation}
		\int
		\Vol(
		\Gamma( \gamma^{(n)} \varphi_{0} )
		)
		d \mu^{n} (\gamma)
		=
		O(1)
		.
	\end{equation}
\end{propo}
Indeed, recall that
the dual Lyapunov exponent
$
\chi_{1}^{*}
$
equals
$
-\chi_{d+1} 
$.

\subsection{Holomorphic variation of the Poisson boundary}
\label{sub:hol_poisson}

In this section we show that the
pluriharmonicity of $ \chi_{1} $ 
allows us to construct
a holomorphic variation of the
Poisson boundary.

The Poisson boundary of the random walk
$ (\Gamma, \mu) $ is a mesurable space
that governs the asymptotic properties
of trajectories of the (right) random walk.

We recall briefly
a construction of the Poisson boundary.
Let $ \g{incr} = \Gamma^{\N} $
the space of increments
with the measure $ \bar{\mu} = \mu^{\N} $.

The natural map from the space
of increment \g{incr}
to the space of trajectories
$ \g{traj} =\Gamma^{\N}$
\begin{align}
	\tau : \g{incr} & \to \g{traj}
	\\
	(\gamma_{n})
	& \mapsto
	(\gamma^{(n)}  = \gamma_{1} \cdots \gamma_{n})
\end{align}
pushes the measure $ \bar{\mu} $
to a measure $ P = \tau_{*} \bar{\mu} $ 
and induces a isomorphism of measured spaces.
The space of trajectories
admits a natural action of $ \Gamma $,
which is coordinate-wise multiplication
on the left.
Both \g{incr}
and \g{traj} 
are equiped with
the coordinate shift
defined by
$ \sigma( (x_{n}))
= (x_{n+1})$
and we have
\begin{equation}
	\tau(\sigma (\gamma_{n}))
	=
	\gamma_{1}^{-1}
	\sigma
	\tau(i)
	.
\end{equation}
We define the following
equivalence relation
on \g{traj} :
\begin{equation}
	t \sim t'
	\iff
	\exists k,k' \in \N,
	\sigma^{k} (t)
	=
	\sigma^{k'} (t')
	.
\end{equation}
As the measure $ P $
is shift-invariant,
we obtain a measure
on the quotient,
still denoted by $ P $.
The
$ \Gamma $-action
also goes down
to the quotient,
and  $ P $
is a $ \mu $-stationary
measure
for this action.
The quotient space
with this measure
is denoted by
$ (\PP, P) $
and is called
the Poisson boundary
of the random walk
$ (\Gamma, \mu) $.
The notation
\say{let $ \xi = (\gamma_{n}) \in \PP $}
means that
$ \xi $ is the element of $ \PP $
corresponding to
$ (\gamma_{n}) \in \g{incr} $ 
via the natural maps,
and $ \gamma^{(n)} $
is the corresponding element
of \g{traj}.
\begin{Remark}
	This definition of the
	Poisson boundary is
	not the usual one
	but it is enough
	for what is needed
	in the following.
	For a general
	definition
	see
	\cite{kaimanovichRandomWalksDiscrete1983}.
\end{Remark}

Consider a
random walk induced by
$ \rho : \Gamma \to G $.
Recall from \ref{ssub:rw_lin}
the Furstenberg
limit map
$ \theta_{\rho} :
\g{incr} \to \P V $.
From its definition,
it is clear that it
factors through
\g{traj} 
and $ \PP $ 
and
induces
a map
$ \theta_\rho : \PP \to \P V $
and that
the stationary measure is
$ \nu_{\rho} = (\theta_{\rho})_{*} P $.
The limit map is $ \rho $-equivariant:
$ \theta_{\rho}(\gamma \xi)
=
\rho(\gamma)
\theta_{\rho}(\xi) $.

Given a holomorphic
family of representation
$ (\rho_{\lambda}) $
for $ \lambda \in U $,
satisfying the standing assumptions,
we define
a \emph{holomorphic variation of
the Poisson boundary}
as a map
\begin{equation}
	\theta:
	\PP \times U
	\to
	\P V
	,
\end{equation}
such that:
\begin{itemize}
	\item $\theta$ is mesurable and defined almost
		everywhere with respect to the
		first variable.
	\item $ \theta $ is holomorphic
		with respect to the second
		variable.
	\item $ \theta $ is equivariant with
		respect to the first variable
		that is
		\begin{equation}
			\theta
			( \gamma \xi,
			\lambda)
			=
			\rho_{\lambda}(\gamma)
			\theta
			(\xi, \lambda)
			,
		\end{equation}
		for almost all $ \xi $
		and all $ \gamma,\lambda $.
\end{itemize}

The goal of this section is to prove
that, if the current
$ T_{1} $ 
is null
on the open set $ U $,
we can construct a holomorphic
variation of the Poisson boundary.
In the following, we make this assumption.

For a given
parameter $ \lambda \in U $,
we define
the map $ \theta( \cdot, \lambda) $
to be $ \theta_{\rho_{\lambda} } $.
For each $ \lambda $,
this map is defined
almost everywhere
on $ \PP $.
The difficulty is to \say{glue} all these
maps together holomorphically,
and to find a full measure subset
on which these maps are defined
for all $ \lambda $.
We follow closely \cite[Lemma 3.10]{deroinRandomWalksKleinian2012}.


The idea of the proof is the following:
We fix a full measure subset of
$ \xi \in \PP $
such that for a countable
dense set of parameters
$ (\lambda_{q} ) $,
$ \theta_{\rho_{\lambda_{q} } } $
is defined at $ \xi $.
Then using the boundedness
in average of the volume
of the graphs
$ \Gamma(\gamma_{1} \cdots \gamma_{n} \sigma_{0}) $
proved in the previous section,
we extend these assignement
to holomorphic maps.
More precisely:

Fix a $ P_{e} $-full measure
subset $\Omega_{1} \subset \PP $ 
such that for a a countable dense set of parameter
$ ( \lambda_{q} )$,
the boundary map
$ \theta_{\rho_{\lambda_{q} } }  $
is defined
on $ \Omega_{1} $,
We will extend these maps defined
only on the $ (\lambda_{q}) $
to holomorphic maps.

For a
$ \xi = (\gamma_{n})
\in \Omega_{1} $,
define $ f_{\xi,n} $ to be the
map
$ \gamma_{1} \cdots \gamma_{n} \sigma_{0} $,
with the notation from the previous
section.
By Theorem \ref{th:bounded_vol},
we know that:
\begin{equation}
	\int
	\Vol(
	\Gamma( f_{\xi,n} )
	)
	d P_{e} (\xi)
	=
	\int
	\Vol(
	\Gamma( \gamma^{(n)} \sigma_{0} )
	)
	d \mu^{n} (\gamma)
	\le
	C
	,
\end{equation}
for some constant $ C $,
that is, the volume
$\Vol(
\Gamma( f_{\xi,n} )
)$
are bounded in mean.
By the Borel-Cantelli lemma we have
\begin{equation}
	\int
	\liminf
	\Vol(
	\Gamma( f_{\xi,n} )
	)
	d P_{e} (\xi)
	\le
	\liminf
	\int
	\Vol(
	\Gamma( f_{\xi,n} )
	)
	d P_{e} (\xi)
	\le
	C
	,
\end{equation}
and consequently
for a $ P_{e} $-full measure
subset $ \Omega_{2} $,
for all $ \xi \in \Omega_{2} $,
$\liminf
\Vol(
\Gamma( f_{\xi,n} )
)$
is finite,
and we can extract a subsequence
$ (f_{\xi, n_{k}}) $
of graphs of bounded volume.
We can apply the following
lemma
(whose proof is postponned
to the end of the section):
\begin{lemma}
\label{lem:bishop}
	If
	$ (f_{i}) $
	is a sequence
	of holomorphic maps
	on $ U $ 
	with graphs
	$ \gamma_{i}
	\subset U \times \P V $
	of uniformly bounded volume,
	then, up to a subsequence,
	the sequence of graphs
	converges (as analytic sets)
	to the union
	of the graph of a holomorphic
	map $ f $ on $ U $ and
	finitely many "bubbles",
	that is analytic sets
	contained in a fiber
	$ z_{0} \times \P V $.

	Moreover,
	away from the bubbles,
	the sequence $ (f_{i}) $ 
	converges
	uniformly on compact subset
	to $ f $.
\end{lemma}
This lemma gives
a holomorphic map
$ f_{\xi} : U \to \P V $
such that 
$ (f_{\xi, n_{k} }) $
converges
(up to a subsequence)
to $ f_{\xi} $
uniformly on compact set
of $ U $ with a finite number
of points removed.
For every $ \lambda_{q} $,
except for a finite number,
we have
$ f_{\xi} (\lambda_{q})
=
\theta(\xi, \lambda_{q})$.
As $ (\lambda_q) $ is dense,
it means that $ f_{\xi} $
is the unique holomorphic
continuation of the assigment
$ \lambda_{q}
\mapsto
\theta(\xi, \lambda_{q}) $.
In particular
$ f_{\xi} $
is
the only cluster value of
$ (f_{\xi, n}) $.
For
$ \xi \in \Omega_{2} $
and
$ \lambda \in U $ 
we can then define
$ \theta(\xi, \lambda)
=
f_{\xi} (\lambda)
$.
It is easily checked
that $ \theta $
satisfies the
required properties.

\begin{proof}[Proof of Lemma \ref{lem:bishop}]
	By a result of Bishop
	(see \cite[15.5]{chirkaComplexAnalyticSets1989}),
	a sequence
	of analytic sets of pure dimension 1
	with uniformly bounded volume
	converges (up to a subsequence)
	to an analytic set of pure dimension 1.
	Denote by $ A $ such an analytic set
	for the sequence $ (\gamma_{i}) $.
	It is of finite volume.

	Pick a fiber $ F = F_{z} = z \times \P V $.
	There is 2 cases: either
	$ A \cap F $ is finite,
	or 
	$ A \cap F = A_{z} $
	is a an analytic set of pure dimension 1.
	In the second case,
	$ A_{z} $ is an algebraic set in $ F = \P V $,
	so its volume is bounded below by
	$ \frac{\pi^{d} }{d!} $.
	In particular, this case can occurs
	for only finitely many $ z $.
	Denote by $ B $ this finite set.

	In the first case by
	the continuity of the intersection index
	(Prop. 2 of 
	\cite[12.2]{chirkaComplexAnalyticSets1989}),
	because the intersection index of
	$ \gamma_{i}$ and $ F $ is 1
	for all $ i $,
	the intersection index of
	$ A $ and $ F $ is 1.
	It means that
	the intersection of $ A $
	and $ F $ is a single point
	with multiplicity 1.
	In particular,
	this point is a point of irreducibility
	for $ A $.
	
	Consider $ (A_{j}) $
	the irreducible components
	of $ A $.
	Each $ A_{z} $
	for $ z \in B $ 
	is a
	union of irreducible
	components, that we call vertical.
	Let $ A_{j} $ be a non-vertical
	component.
	For some $ z_{0} \notin B $,
	$ A_{j} \cap F_{z_{0} } $ is not empty,
	so it is a singe point.
	The intersection index
	of $ A_{j} $ and $ F_{z_{0} } $
	depends only on the homology class
	of $ F_{z_{0} } $, in particular
	it doesn't depend on $ z_{0} $,
	so when 
	$ A_{j} $ and $ F_{z} $
	intersect at isolated points,
	their intersection is a single points.
	In fact $ A_{j} $ always intersects $ F_{z} $ 
	at isolated points: otherwise if would be
	completly contained in $ F_{z} $ and
	would be a vertical component.

	This imply that there is only one
	non-vertical irreducible component
	that we denote by $ A_{0} $.
	We have seen that it intersects every
	fiber transversally at a single point.
	By \cite[3.3, Prop. 3]{chirkaComplexAnalyticSets1989},
	$ A_{0} $ is the graph of a holomorphic map
	$ f $.
	By the definition of convergence of analytic
	set (see \cite[15.5]{chirkaComplexAnalyticSets1989}),
	the sequence $ (f_{i}) $
	converges
	to $ f $,
	uniformly
	on compact set not meeting $ B $.
\end{proof}

By using the pluriharmonicity of $ \chi_{d+1} $
and doing the same construction as in the previous
section we obtain a  holomorphic variation of
the Poisson boundary $ \PP^{*} $
associated to the dual random walk
in $ \P V^{*} $
denoted by:
\begin{equation}
	\theta^{*}:
	\PP^{*} \times U
	\to
	\P V^{*} 
,
\end{equation}
equivariant with respect
to the dual representation
$ \rho^{*} $ 
.

\subsection{Intersections of curves and hyperplanes}
\label{sub:intersections}

In this section we prove:
\begin{propo}
\label{prop:intersection}
	There exists
	a full measure subset
	$ \mathcal{D} \subset
	\PP \times \PP^{*} $
	such that for
	$ (\xi, \xi^{*}) \in \mathcal{D} $ 
	the graphs
	$ \theta_{\xi} $
	and $ \theta_{\xi^{*} }^{*} $ 
	don't intersect.
\end{propo}

Fix a relatively compact open subset
$ D $ inside $ U $.
For
$ (\xi, \xi^{*})
\in
\PP \times \PP^{*} $
define
$ i_{D} (\xi, \xi^{*}) $
as the number of intersection points
of the graphs
$ \theta_{\xi} $
and $ \theta_{\xi^{*} }^{*} $
in $ D $.
Remark that either
$ \theta_{\xi} $ 
intersects $ \theta^{*}_{\xi^{*} } $
in a finite number of points
in $ D $
or
$ \theta_{\xi} $ 
is entirely contained in
$ \theta^{*}_{\xi^{*} } $.
Because $ \nu_{\lambda} $
don't charge
any linear subspace,
for almost all $ (\xi, \xi^{*}) $
the latter doesn't happen
and $ i_{D} (\xi, \xi^{*}) $ is finite.

As $ \theta $ and $ \theta^{*} $
are equivariant,
$ i_{D} $ is invariant by the diagonal
action of $ \Gamma $.
By the double ergodicity theoreom
of Kaimanovich
\cite{kaimanovichDoubleErgodicityPoisson2003},
this function is constant almost surely,
equal to a constant we still denote by
$ i_{D} $.

The map $ D \mapsto i_{D} $ extends to
an integer-valued measure on $ U $,
which is finite on relatively compact subsets.
It follows that it is a sum of Dirac masses
with integer coefficients, supported on a
discrete set $ F $.
It means that
there exists a full measure subset
$ \mathcal{D} \subset \PP \times \PP^{*} $
such that
for all $ p \in U \setminus F $,
if $ (\xi, \xi^{*}) \in \mathcal{D} $
then $ \theta_{\xi} $
and
$ \theta^{*}_{\xi^{*}} $
don't intersect.

We show that this set $ F $ is empty.
Let $ p \in F $.
For almost all
$ (\xi,\xi^{*})
\in
\PP \times \PP^{*} $
the graphs $ \theta_{\xi} $
and $ \theta_{\xi^{*} } $ intersect
at $ p $.
Fix a $ \xi^{*} $
and a full measure subset
of
$  \xi \in \Omega \subset \PP $
such that this holds.
Then the image
$ \theta( p, \Omega) $ 
is contained in
the kernel of
$ \theta^{*}(p, \xi^{*}) $.
Because
the stationary measure
$ \nu_{p} $ verifies
$ \nu_{p} =
\theta(p, \cdot)_{*}  P_{e} $,
it implies that the support
of $ \nu_{p} $ is contained
in the kernel of
$ \theta^{*}(p, \xi^{*}) $,
which is absurd because
the stationary measure
doesn't charge linear subspaces.
We have proved the proposition \ref{prop:intersection}.

\subsection{Compactness of the holomorphic variation}
\label{sub:compact_hol_var}

The goal of this section
is to prove Proposition \ref{prop:compact_var}
which asserts that the
maps $ \theta_{\xi} $ 
form a normal family
when $ \xi $ varies
in some full measure subset.

We will use a technical lemma
which asserts that some
space of holomorphic maps
avoiding some hyperplanes is compact.
It is just a slight variation
on a classical result in
complex hyperbolic space,
but we need to detail
some constructions related
to hyperplanes in projective space.

\subsubsection{Some technical results about hyperplanes}
\label{ssub:hyperplanes}

In order to use a result
of \cite[Th. 3.10.27]{kobayashiHyperbolicComplexSpaces1998}
we need to study hyperplanes in
general position.
We first explain what this means,
then how to construct such family
and finaly extend these
definition to graphs
of hyperplanes.

We say that a family of hyperplanes
is in general position if
any $ (d+1) $ of them are
linearly independent.
Let $ (H_{i})_{i=0, \dots, d+1} $ be
a family of $ d+2 $ hyperplanes in
general position.
We call such a family a system of hyperplanes.
There exists a (unique, up to homotheties)
basis of $ V $,
called the normal basis,
such that in the
coordinates of this basis, the hyperplanes
are defined by the equations:
\begin{align}
	(H_{0})&:
	x_{0} = 0 \\
	&\cdots \\
	(H_{d})&:
	x_{d} = 0 \\
	(H_{d+1})&:
	x_{0}+ \dots + x_{d} = 0
	.
\end{align}
We say that the system is in normal form.

The diagonals associated to this system
are the hyperplanes
$ \delta_{I} $, where $ I $ is a subset
of $ \{0, \dots, d \} $ which is not
a singleton neither the entire set,
defined in the normal basis by:
\begin{equation}
	(\delta_{I}):
	\sum_{i \in I}
	x_{i} 
	=0.
\end{equation}
We will also denote $ H_{d+1} $ by
$ \Delta $.

If we work in the dual projective space,
and consider an hyperplane as a point
of the dual projective space,
we remark the following relations:
for all subset $ I $,
$ \delta_{I} $
is the intersection
of the space spanned
by the 
$((H_{j})_{j \notin I}
, \Delta)$
with the space spanned
by the
$(H_{i})_{i \in I}$.
Another way to say it
is that
$\delta_{I}$
is the image of $ \Delta $
by the projection to
$span (H_{i})_{i \in I} $
parallel to
$span (H_{k})_{k \notin I}$.

Now, let's consider
$ (S^{1}, \dots, S^{k}) $
a family of systems of hyperplanes.
We say that this family is
in general position if
the family of all associated
hyperplanes and all the associated
diagonals are in general position.
We can construct such family
in the following generic way:

\begin{lemma}
\label{lem:generic_systems}
	Given an integer $ k $
	we can construct
	by induction
	a family of systems of hyperplanes
	$ (S^{1}, \dots, S^{k}) $,
	and $ S^{j} = (H_{i}^{j})_{i=0, \dots, d+1} $,
	in general position, in the following way:
	Choose any $ H_{1}^{1} \in \P V^{*} $.
	Suppose we have constructed
	the hyperplanes up to some $ H_{i}^{j} $.
	Then for the next hyperplane
	$H_{i+1}^{j}$ 
	or
	$H_{1}^{j+1}$,
	we can choose any hyperplane
	in $ \P V^{*} $ that is not
	in an exceptional set,
	which is a finite union of proper linear subspaces.
\end{lemma}
\begin{proof}
	We choose any $ H_{1}^{1} \in \P V^{*} $.
	Suppose we have constructed
	the hyperplanes up to some $ H_{i}^{j} $,
	with an exceptional set $ E_{ij} $.

	If $ i<d $, we can choose for $ H_{i+1}^{j} $
	any hyperplane which is in general position
	with all the hyperplanes already constructed
	and the diagonals hyperplanes associated to the systems
	already constructed.
	So we add to the exceptional set the linear
	spaces spanned by any $ d $ of theses hyperplanes.
	The same applies for $i=d+1$ to construct
	$ H_{1}^{j+1} $.

	If $ i=d $, we have to choose
	$ H_{d+1}^{j} = \Delta^{j} $.
	Fix an admissible subset
	of indices $ I $.
	The choice of $ \Delta^{j} $ determines
	the diagonal $ \delta_{I}^{j} $,
	explicitely
	$\delta_{I}^{j}$
	is is image of $ \Delta^{j} $
	by the projection to
	$span (H_{i}^{j})_{i \in I} $
	parallel to
	$span(H_{k}^{j})_{k \notin I}$.
	We want $ \delta_{I}^{j} $
	to be in general position with all
	the already constructed hyperplane,
	that is to avoid the spaces spanned
	by any $ d+1 $ of them.
	These spaces intersect
	$span(H_{i}^{j})_{i \in I} $
	tranversally, because
	the previous hyperplanes
	and diagonals are
	all in general position.
	So the inverse images of the spaces
	by the projection are proper subspaces
	and we just have to add them to the exceptional set.
\end{proof}

Remark that this lemma
imply that being a system,
or a family of systems of hyperplanes
is an open condition
in the appropriate space.

\subsubsection{Compactness of maps avoiding varying hyperplanes}
\label{ssub:compactness}

Let $ U $ be the unit disk in $ \C $.
A graph of hyperplane $ H $
(above $ U $)
is a holomorphic map $ H: U \to \P V^{*} $
where we confuse $ H(\lambda) $ and the
hyperplane it defines in $ \P V $.
We define in the same way
graph of linear subspaces of any dimension.

We extend all the definitions of
the section \ref{ssub:hyperplanes}
about hyperplanes
to graphs of hyperplanes in the obvious way.

Given a graph of subspaces $ K $, we say
that a sequence $ f_{n} $ of maps $ U \to \P V $
converges to $ K $ if the image of any compact
by $ f_{n} $ 
is contained in any neighborhood of $ K $
for $ n $ large enough.

We say that a subspace $ F $ of $ Hol(U, \P V ) $
is relatively compact modulo a graph of subspaces
$ K $ if any sequence of elements of $ F $ is either
convergent up to a subsequence, or convergent to $ K $.

\begin{lemma}
	\label{lem:hyper_compact}
	Let
	$ (H_{k})_{k=0,\dots,d+1} $,
	be a graph of
	systems of hyperplanes
	above
	$ U $.
	Let $ F $ be a set of holomorphic
	maps $ U \to \P V $.
	Assume that any $ f \in F $
	avoids all the $ H_{k} $.
	Then $ F $ is relatively compact
	in $ Hol(U, \P V) $ modulo the
	diagonals hyperplanes of $ (H_{k}) $.
\end{lemma}
\begin{proof}
	For each $ \lambda \in U $,
	because the $ H_{k} (\lambda) $
	are in general position,
	we can find
	a projective automorphism
	$ \Phi_{\lambda} $,
	(which vary holomorphically with $ \lambda $)
	which put the system $ (H_{k}(\lambda)) $ in normal form.
	By replacing
	$ F $ by
	$ \left\{
	\Phi_{\lambda} f;
	f \in F
	\right\} $
	and the $ H_{k} $
	by the $ \Phi_{\lambda} H_{k} $,
	we can consider
	that the graphs of hyperplanes
	are constant.

	The theorem \cite[Th. 3.10.27]{kobayashiHyperbolicComplexSpaces1998}
	states the space of holomorphic maps
	$ U \to \P V $ avoiding $ d+2 $ hyperplanes
	in general position is relatively compact modulo the
	diagonals.
	Composing everything with $ \Phi_{\lambda}^{-1} $,
	we get the result.
\end{proof}

Now we can prove:

\begin{propo}
\label{prop:compact_var}
	There exists a full measure subset
	$ \Omega $ of the Poisson boundary $ \PP $ 
	such that the set
	of holomorphic maps
	\begin{equation}
		\Theta
		=
		\left\{ 
			\theta(\xi, \cdot);
			\xi \in \Omega
		\right\}
		,
	\end{equation}
	is relatively compact in
	$ Hol(U_{0}, \P V) $,
	where $ U_{0}  $
	a Zariski dense open set
	of $ U $.
	Moreover, $ U_{0} $
	can be chosen to contain
	any relatively compact
	subset of $ U $.
\end{propo}

\begin{proof}
	Choose a
	$ \lambda_{0} \in U$.
	The stationary measure 
	$ \nu_{\lambda_{0} }^{*} $ 
	at the parameter $ \lambda_{0} $ for the dual
	representation doesn't charge any linear subspace
	of $ \P V^{*}  $.

	By using Lemma \ref{lem:generic_systems}
	we can construct a family of systems
	in general position
	$ (S^{1}, \dots, S^{d+1}) $ 
	with every hyperplanes of these
	systems in the support of $ \nu_{\lambda_{0} }^{*} $,
	because it doesn't charge linear subspaces.

	Because the image of
	$ \theta^{*} (\PP^{*}, \lambda_{0}) $
	is dense in the support of
	$ \nu_{\lambda_{0} }^{*} $,
	we can find trajectories such that their
	traces at $ \lambda_{0} $
	are close enough to the hyperplanes
	of
	$ (S^{1}, \dots, S^{d+1}) $,
	and so they also form a family
	of systems of hyperplanes in general position.

	We fix such trajectories,
	$ \xi_{i}^{j*} \in \PP^{*} $.
	Let $ U_{0} $ be the subspace
	such that for all $ \lambda \in U_{0}  $ the family
	of systems
	$ \theta^{*} (\xi_{i}^{j*}, \lambda) $ is still
	in general position.
	By the result of
	Lemma \ref{lem:generic_systems},
	the complement of $ U_{0} $ is an analytic
	set of positive codimension,
	so $ U_{0} $ is a Zariski dense open set of $ U $.

	By the Proposition \ref{prop:intersection},
	there exists a full measure subset
	$ \Omega $ of the Poisson boundary $ \PP $,
	such that the graphs $ \theta(\xi, \cdot) $
	doesn't intersect the graphs of hyperplanes
	$ \theta^{*} (\xi_{i}^{1*}, \cdot) $
	over $ U_{0} $,
	for $ \xi \in \Omega $.

	We can now apply the Lemma \ref{lem:hyper_compact}
	to the system of graphs of hyperplanes
	$ (\theta^{*}_{\xi_{i}^{1*}}) $.
	We deduce that the family
	$ (\theta_{\xi}))_{\xi \in \Omega} $ 
	is relatively compact
	modulo the diagonal
	$ \Delta^{1} $.

	We show now that it is
	relatively compact (not only modulo the diagonal).
	Consider a sequence $ \theta_{\xi_{k}} $
	which converges to a diagonal graph
	$ \delta^{1}  $.
	We apply again Lemma \ref{lem:hyper_compact},
	but now to
	$ (\theta^{*}_{\xi_{i}^{2*}}) $.
	Then this sequence is also relatively
	compact modulo the diagonal
	$ \Delta^{2} $ associated to
	this new system.
	Now $ \theta_{\xi_{k}} $
	is either convergent (up to a subsequence)
	or convergent to a $ \delta^{2} $,
	and so convergent to
	$ \delta^{1} \cap \delta^{2} $,
	which is a graph of linear subspaces of codimension 2
	(by the general positions of the diagonal hyperplanes).
	Iterating the argument at most $ d+1 $ times,
	either the sequences is convergent up to a
	subsequence or it must converges
	to an empty set, which is impossible.

	We prove that $ U_{0} $
	can be chosen to contain
	any relatively compact
	subset of $ U $.
	Given a relatively compact
	subset $ K \subset U $,
	the difference $ K \setminus U_{0} $
	consists of finitely many points.
	Arround each of these points $ \lambda $ ,
	repeat the previous argument
	with $ \lambda $ instead of $ \lambda_{0} $.
	We obtain a full measure subset $ \Omega $
	such that 
	$ 
	(\theta_{\xi})_{\xi \in \Omega}
	$ 
	is relatively compact over
	a neighborhood of each of these points,
	so over $ K $.
\end{proof}

All the results of this section can
be done dually by exchanging the role
of $ \theta $ and $ \theta^{*} $.
Consequently we can state
the dual version of
proposition \ref{prop:compact_var}:
\begin{propo}
\label{prop:compact_var_dual}
	There exists a full measure subset
	$ \Omega^{*} $
	of the dual Poisson boundary
	$ \PP^{*}  $ 
	such that the set
	of holomorphic maps
	\begin{equation}
		\g{graphs_dual} 
		=
		\left\{ 
			\theta^{*}
			(\xi^{*} , \cdot);
			\xi^{*} \in \Omega^{*} 
		\right\}
		,
	\end{equation}
	is relatively compact in
	$ Hol(U_{0}, \P V^{*} ) $,
	where $ U_{0} $
	a Zariski dense open set
	of $ U $.
\end{propo}
We can, and will, fix a set $ U_{0} $
which works for proposition
\ref{prop:compact_var}
and proposition \ref{prop:compact_var_dual}.

Recall that the proposition
\ref{prop:intersection}
gives
$ \mathcal{D} \subset \PP \times \PP^{*} $
of full measure
such that
for
$ (\xi,\xi^{*})
\in
\mathcal{D}$,
the graph
$ \theta_{\xi} $ 
and
the graph of hyperplanes
$ \theta^{*}_{\xi^{*}} $
don't intersect.
We will show
that
an analogous
non-intersection
property
also holds for
every couples of
$ \g{graphs}  \times \g{graphs_dual} $,
not only almost surely.

We consider the closure
$ \bar{\Theta} $
of $ \Theta $ in
$ Hol(U_{0}, \P V) $.
By the proposition
\ref{prop:compact_var},
the closure is compact.
Then, pushing the probability
$ P $
by $ \theta $,
\g{graphs}
is equiped with a
probability measure.
We replace \g{graphs}
by the support of this
measure.
The same can be done
for \g{graphs_dual}
and we have that
for almost all
$ (f, H) \in \g{graphs} \times \g{graphs_dual} $,
$ f $ and $ H $
don't intersect.

\begin{lemma}
\label{lem:intersection}
	Let
	$ (f, H) \in  
	\g{graphs}  \times \g{graphs_dual}  $
	.
	Then either $ f $
	and $ H $ are disjoints,
	or $ f $ is entirely
	contained in $ H $.
\end{lemma}
\begin{proof}
	Consider
	$ (f, H) \in  
	\g{graphs} \times \g{graphs_dual}  $
	.
	Suppose that 
	the graph of $ f $
	and the hyperplane graph
	of $ H $ admit a non-trivial
	isolated intersection.
	By the continuity of intersection
	of analytic set
	\cite[sec. 12.3]{chirkaComplexAnalyticSets1989},
	if $ f $ and $ H $
	admit an isolated intersection
	then their exists a neighborhood $ N $
	of $ (f,H) $ 
	in $ \g{graphs} \times \g{graphs_dual} $
	of couple which
	also admit a non-trivial isolated intersection.

	But as $ (f,H) $ is in the support
	of the measure,
	any of its neighborhood has positive measure.
	This contradicts
	the fact that the set non-intersecting
	couples is of full measure.

	We deduce that if $ f $
	and $ H $ intersect
	then the graph of $ f $ 
	is entirely included
	in the graph of $ H $.
\end{proof}

The compactness of
$ \overline{\Theta} $ allows
many constructions,
as in the following lemma.

\begin{lemma}
\label{lem:lot_of_graph}
	Let $ \lambda_{0} \in V $
	and $ z_{0} \in \supp \nu_{\lambda_{0} }  $.
	Then there exists a
	$ f \in \overline{\Theta} $ 
	such that
	$ f(\lambda_{0}) = z_{0} $
	.
\end{lemma}
\begin{proof}
	The image
	of $ \theta_{\lambda_{0}} $ 
	is dense in $ \supp \nu_{\lambda_{0} } $.
	Choosing $ \xi_{k} $ such that
	$ \theta(\lambda_{0}, \xi_{k}) \to z_{0} $,
	we can choose $ f $ to be
	any cluster value of the sequence
	$ (\theta_{\xi_{k}}) $
	in $ \overline{\Theta} $.
\end{proof}
An analogous lemma holds dually.

We summarize the results of the
previous sections
\ref{sub:hol_poisson},
\ref{sub:intersections},
\ref{sub:compact_hol_var}:

\begin{propo}
	\label{prop:summarize}
	If $ (\rho_{\lambda})_{\lambda \in \Lambda} $ 
	satisfies the standing assumption
	and $ \chi_{1} $ and $ \chi_{d+1} $ 
	are harmonic
	on $ \Lambda $ 
	then
	there exists
	two family of holomorphic maps
	\begin{align}
		\g{graphs} & \subset Hol(\Lambda, \P V),
		\\
		\g{graphs_dual} & \subset Hol(\Lambda, \P V^{*} )
		,
	\end{align}
	which are normal
	and which are transverse:
	for any
	$ f \in \g{graphs} $ 
	and
	$ H \in \g{graphs_dual}  $,
	either $ f $ and $ H $
	don't intersect
	or $ f $ is completely contained in $ H $.
\end{propo}
This construction is central to the present work.
We can compare these family of graphs
with the \emph{equilibrium web}
and \emph{equilibirum lamination}
constructed in \cite{bertelootDynamicalStabilityLyapunov2018},
which play similar roles in the proof.
It seems that the construction is easier in our case,
due to the interplay between graphs and dual graphs,
which does not have an analog in the theory of endomorphisms.
They also defines \emph{branched holomorphic motions},
in the sense of
\cite{bergerStabilityHyperbolicityPolynomial2017},
and play a similar role.

\subsection{The attracting fixed points move holomorphically}
\label{sub:fixed_points_hol}

In this section we first show that a proximal element
at some parameter $ \lambda_{0} $ stays proximal
on the whole domain $ U_{0} $ given by
Proposition \ref{prop:compact_var},
then on all of $ U $.

The proof is in two step.
First we show that if $ \gamma $ is proximal at $ \lambda_{0} $
then we can find a graph $ f $  in $ \overline{\Theta} $
which is fixed by $ \rho_{\lambda} (\gamma) $
for all $ \lambda \in U_{0} $, and which is the attractive
fixed point of $ \rho_{\lambda_{0} } (\gamma) $
at $ \lambda_{0} $.
Then we show that this graph of fixed point
is in fact a graph of attractive fixed points,
that is there exists a graph of hyperplane
$ H $ 
in $ \overline{\Theta^{*}}$ such 
at every $ \lambda $, $ \rho_{\lambda} (\gamma) $
contracts every point not in $ H(\lambda) $ to $ f(\lambda) $.

\subsubsection*{Graph of fixed points}
\label{ssub:graph_fp}
Let's fix a parameter $ \lambda_{0} \in U_{0} $ 
and a group element $ \gamma $
such that $ \gamma $ is proximal at
$ \lambda_{0} $.
For $ \lambda $ in
a small neighborhood $W$ of $ \lambda_{0} $,
the matrix $ \rho_{\lambda }(\gamma) $
stays proximal.
Let $ f : W \to \P V $
be the graph of attractive fixed points
and
$ H : W \to \P V^{*} $
be the graph of repulsive fixed hyperplane of
$ \rho (\gamma) $ on $ W $ 
.
The goal is to extend
these graphs to all of $ U_{0} $.

By lemma \ref{lem:lot_of_graph},
there exists a graph
$ g \in \overline{\Theta} $ 
such that $ g(\lambda_{0}) = f(\lambda_{0}) $
and a graph of hyperplane
$ K \in \overline{\Theta^{*} } $ 
such that $ K(\lambda_{0}) = H(\lambda_{0}) $,
because an attractive fixed point
belongs to the support of the stationary measure.
We're going to show that $ f=g $ and $ H=K $
on $ W $.

Suppose, by contradiction,
that $ f \neq g$ in $ W $.
Shrinking $ W $ if necessary
we can assume that $ f $ and $ g $
intersect only at $ \lambda_{0} $.
Let $ g_{n} = \gamma^{-n} g $.
The sequence $ (g_{n}) $ is in
$ \overline{\Theta} $,
so it has a cluster value $ g_{\infty} $.
We have
$ g_{n} (\lambda_{0})=
g(\lambda_{0})
=
g_{\infty} (\lambda_{0}) $
and
$
g_{n} (\lambda) =
\rho_{\lambda}
(\gamma^{-n})
g(\lambda)
$ 
which converges to a point in $ H(\lambda) $,
so $ g_{\infty} (\lambda) \in H(\lambda) $,
for all $ \lambda \neq \lambda_{0} $.
This is impossible
because
$ g(\lambda_0)
= f(\lambda_{0})
\notin H(\lambda_{0})$
and we have obtained
the desired contradiction.
We must have $ f = g $ in $ W $.
Dually, by the same reasoning
we have that $ H = K $ in $ W $.

As $ g $ and $ K $ are defined
on all of $ U_{0} $,
we have extended $ f $ and $ H $
to $ U_{0} $.
By analytic continuation,
$ f $ stays a graph of fixed points
and $ H $ stays a graph of fixed
hyperplanes.

\subsubsection*{Graph of attractive fixed points}
We need to show that $ f $ and $ H $ stays
attractive fixed points
and repulsive hyperplanes on all of $ U_{0} $,
that is for all $ \lambda $,
for all
$ z \in 
\P V \setminus
H(\lambda) $,
we have
$ \rho_{\lambda}
(\gamma^{n})
z
\to f(\lambda) $.
In fact, it is enough
to show it for only one
generic $ z \in \P V $;
more precisely, suppose
that $ z \in \P V $ has
all of its coordinates non zero
in a basis adapted to the direct
sum
$ f(\lambda) \bigoplus H(\lambda) $.
Then
$ \rho_{\lambda}
(\gamma^{n})
z
\to f(\lambda) $
implies that the spectral radius
of
$ \rho_{\lambda}
(\gamma)_{H(\lambda)}$
is strictly smaller
than the maximal eigenvalue
of 
$ \rho_{\lambda}
(\gamma)$,
and so that
$ \rho_{\lambda}
(\gamma)$
is proximal.

Fix
some $ \lambda_{1} \in U_{0} $.
As the stationary measure
$ \nu_{\lambda_{1} }  $
doesn't charge linear subspaces,
there is a generic
$ z \notin H(\lambda_{1})  $
in its support.
Suppose that
we don't have
$ \rho_{\lambda_{1} }
(\gamma^{n})
z
\to f(\lambda_{1} ) $.
By the lemma \ref{lem:lot_of_graph},
there exists a graph
$ j \in \overline{\Theta} $ 
such that $ j(\lambda_{1})=z $.
As $ j(\lambda_{1}) \notin H(\lambda_{1}) $,
the graph $ j $ and $ H $ are disjoints.
The sequence $ \gamma^{n} j $ 
has a cluster value $ j_{\infty} $.
For all $ \lambda \in W $,
we have 
$ \gamma^{n}
j(\lambda)
\to f(\lambda) $
because $ j(\lambda) \notin H(\lambda) $
and $ f $ is a graph of attracting fixed
points over $ W $,
so
$ j_{\infty} (\lambda) 
= f(\lambda) $
for all $ \lambda \in W $.
By analytic continuation we must have
$ j_{\infty} = f$ on all of $ U_{0} $.
It implies that
$ j_{\infty} (\lambda_{1}) = f(\lambda_{1}) $ 
but that is impossible because
$ j_{\infty} (\lambda_{1})$ 
is a cluster value of
$ \rho_{\lambda_{1} }
(\gamma^{n})
z$.
We conclude that we have
the convergence
$ \rho_{\lambda_{1} }
(\gamma^{n})
z
\to f(\lambda_{1} ) $
for this generic $ z $,
and so for all $ z \notin H(\lambda_{1}) $.

\emph{Alternative proof.}
Let $ \mu_{1} $ be the log of the modulus
of the eigenvalue of $ \rho_{\lambda} (\gamma) $
at $ f(\lambda) $
and $ \mu_{2} $ be the log of the
spectral radius of $ \rho_{\lambda} (\gamma) $
on $ H(\lambda) $.
The map $ \mu_{1} $ is pluriharmonic
and the map $ \mu_{2} $ is plurisubharmonic,
so $ \mu_{1} - \mu_{2} $
is a plurisurharmonic non-negative map.
If it vanishes somewhere, it is identically null,
but it is not zero at $ \lambda_{0} $ by assumption.
Consequently we have $ \mu_{1} > \mu_{2} $ everywhere
and $ \rho_{\lambda} (\gamma) $ stays proximal on $ U_{0} $.

\section{Propagation}
\label{sec:stab_consequences}

In this section we prove
the theorem \ref{th:progation}.
This theorem is about \emph{propagation}:
we assume that a property is satisfied
at some parameter
and we prove that it is satisfied
at every parameters.
In the following proofs
the main tools is the existence
of compact and transverse
families of graps
\g{graphs} 
and
\g{graphs_dual},
see proposition
\ref{prop:summarize}.
We fix a family
$ (\rho_{\lambda}) $
satisfying the standing
assumptions
and such that
$ \chi_{1} $
and $ \chi_{d+1} $
are harmonic,
and equivalently,
by theorem \ref{th:main},
which is proximally stable.

\subsection{Faithfulness}
\label{sub:faithfulness}

Suppose that for some parameter
$ \lambda_{0} $
the representation is
faithful.
We prove that
for every $ \lambda $,
$ \rho_{\lambda} $
is faithful.

Fix a doubly-proximal element,
that is a $ \gamma_{0}  \in \Gamma $,
such that
$ \rho_{\lambda_{0} } (\gamma_{0} ) $
and
$ \rho_{\lambda_{0} } (\gamma_{0} )^{-1} $
are proximal
(such an element exist by 
\cite{guivarchProduitsMatricesAleatoires1990a}).
Fix a $ h \in \Gamma $
which is not the identity.
We want to show that
we can't have $ \rho_{\lambda_{1} } (h) = id $
for some $ \lambda_{1} \in U $.

Let $ \gamma_{1} = h \gamma_{0} h^{-1} $.
It is proximal at $ \lambda_{0} $,
so it stays proximal on all of $ U $.
We show that the graphs of
fixed points of $ \gamma_{0} $
and $ \gamma_{1} $ are disjoint
(possibly by changing $ \gamma_{0} $)
so the element $ h $ can't be
mapped to the identity at any
parameter.
For a proximal element $ \gamma $ 
(it doesn't depend on the parameter by assumption)
we denote by
$ Fix^{+} \gamma$
its graph of attractive fixed points
and by
$ Fix^{-} \gamma$
its hyperplane graph of repulsive hyperplane.

We have
$ Fix^{+} \gamma_{0}
\subset
Fix^{-} \gamma_{0}^{-1} $,
because $ \gamma_{0} $ and $ \gamma_{0}^{-1} $ 
are proximals.
So if
$ Fix^{+} \gamma_{1} $
avoids
$ Fix^{-} \gamma_{0}^{-1} $,
we're done.
If not, it means that
$ Fix^{+} \gamma_{1} $
is entirely contained in
$ Fix^{-} \gamma_{0}^{-1} $
by Lemma \ref{lem:intersection},
ie
\begin{equation}
	h( 
	Fix^{+} \gamma_{0}
	)
	\subset
	Fix^{-} \gamma_{0}^{-1}
	.
\end{equation}

We now work
at the parameter
$ \lambda_{0} $,
but we ommit it in the notation.
Denote
$
L = 
Fix^{+} \gamma_{0}
$
and
$
H=
Fix^{-} \gamma_{0}^{-1}
$,
so we have
$ h(L) \subset H $.
If we replace
$ \gamma_{0} $
by a conjugate
$ g \gamma_{0} g^{-1} $,
then $ L $ and $ H $
are replaced by
$ gL $ and $ gH $.
We show that for some $ g $,
the relation $ h(gL) \subset gH $
doesn't hold.

The existence of a doubly-proximal
element implies that the group
$ \rho_{\lambda_{0} } (\Gamma) $
is proximal on the flag variety
$ \FF = \left\{ (L,H); L \subset H \right\} $.
By results of \cite{benoistRandomWalksReductive2016}
there is a unique stationary measure $ \nu_{\FF} $
on $ \FF $. This measure doesn't charge any subvariety
of $ \FF $, in particular the subvariety
$ \left\{ (L,H); h(L) \subset H \right\} $.
As almost every orbit $ g_{n} \cdots g_{1} (L,H) $
equidistributes towards $ \nu_{\FF} $,
for some $ g = g_{n} \cdots g_{1} $ a big enough
random product, we don't have $ h(gL) \subset gH $.

For such a $ g $, replacing $ \gamma_{0} $
by $ g \gamma_{0} g^{-1} $, we have that
$ Fix^{+} \gamma_{1} $
is not entirely contained in
$ Fix^{-} \gamma_{0}^{-1} $,
because it's false at parameter $ \lambda_{0} $,
and so
$ Fix^{+} \gamma_{1} $
and
$ Fix^{-} \gamma_{0}^{-1} $
don't intersect.
Finally, the graphs
$ Fix^{+} \gamma_{0} $
and
$ Fix^{+} \gamma_{1} $
don't intersect,
and the element $ h $
can't be trivial at any parameter.

\subsection{Discreteness}

Suppose that for some parameter
$ \lambda_{0} $
the representation is
discrete.
We prove that for all parameters
$ \lambda $ the representation
$ \rho_{\lambda } $ 
is discrete.

Fix a sequence $(h_{n})$
of elements of $ \Gamma $
going to infinity in $ \Gamma $.
We're going to show that
$ \rho_{\lambda_{1}} (h_{n}) $
can't converge to the identity,
for all $ \lambda_{1} $.

As $ \rho_{\lambda_{0} } $ is discrete,
the sequence $ \rho_{\lambda_{0} } (h_{n}) $
goes to infinity.
By replacing $ (h_{n}) $ with a subsequence,
we can assume that
it converges (up to scalar) to an endomorphism
$ \pi $ with range $ R $ and (non-trivial) kernel $ K $.


The idea of the proof is the following:
let $ \gamma_{n} = h_{n} \gamma_{0} h_{n}^{-1} $
for a carefully chosen doubly-proximal $ \gamma_{0} $.
We show that the graphs $ Fix^{+} h_{n} $ converges
to a graph $ f $ that is disjoint from
$ Fix^{+} \gamma_{0} $.
It readily implies that 
$ \rho_{\lambda_{1} } (h_{n}) $ 
can't converge to the identity,
for all $ \lambda_{1} $,
as it would cause an intersection
at $ \lambda_{1} $ 
between 
$ Fix^{+} \gamma_{0} $
and $ f $.

\subsubsection*{First step}
Like in the previous section \ref{sub:faithfulness},
for each $ n $, we can choose $ \gamma_{0} $
so that
$ Fix^{+} \gamma_{n} $
avoids
$ Fix^{-} \gamma_{0}^{-1} $
.
The first step
is to show that we can choose a $ \gamma_{0} $
that works for all $ n $.

Like before, the graph
$ Fix^{+} \gamma_{n} $
is disjoint from
$ Fix^{-} \gamma_{0}^{-1} $
if the former in not
entirely contained in the latter.
We show that at the parameter $ \lambda_{0} $
they don't intersect.
We fix everything at $ \lambda_{0} $,
but we ommit it in the notation.
Let $ L = Fix^{+} \gamma_{0} $,
and $ H = Fix^{-} \gamma_{0}^{-1} $.
As we have
$ h_{n} L =
Fix^{+} \gamma_{n}$ 
we want to show that
we can choose $ \gamma_{0} $
such that
$ h_{n} L \notin H $.
With the notation from \ref{sub:faithfulness},
we want that $ (L,H) \notin \FF_{h_{n} } $.
If we replace $ \gamma_{0} $ by $g \gamma_{0} g^{-1}$,
then $ (L,H) $ is replaced by $ (gL, gH) $.
We have to find a $ g $ such that
$ (gL, gH) $ is not in
$ \bigcup \FF_{h_{n} } $.

Let $ V = 
\bigcup \FF_{h_{n} }
\cup
\FF_{\pi}
$,
where $ \FF_{\pi} $ is
the subset of $ \FF $
composed by the elements
$ (L,H) $ such that
either $ L \in K $
or $ L \notin K $
and $ \pi L \in H $.
Then $ V $ is a closed subset of $ \FF $
and is a measure $ 0 $ subset for
$ \nu_{\FF} $.
Now, for almost all large enough random product
$ g = g_{1} \cdots g_{n} $,
the orbit 
$ g (L,H) $
is not in $ V $.

Replacing $ \gamma_{0} $
by $ g \gamma_{0} g^{-1} $,
we have that
$ h_{n} L \notin H $
for all $ n $,
and consequently
$ Fix^{+} \gamma_{n} $
is disjoint from
$ Fix^{-} \gamma_{0}^{-1} $,
because they don't
intersect at the parameter
$ \lambda_{0} $.

\subsubsection*{Second step}
By compactness,
the sequence of graphs
$ Fix^{+} \gamma_{n} $
converges (up to replacing it with a subsequence)
to a graph $ f $.
We proved that all these graphs
avoids $ Fix^{-} \gamma_{0}^{-} $.
By continuity of the intersection,
the limit $ f $ is either disjoint
from it
or entirely contained in it.
We will prove that at the parameter $ \lambda_{0} $
they don't intersect,
and consequently that they are disjoint.

The value of $ f $ at $ \lambda_{0} $
is the point $ P $ which is
the limit of
$ h_{n}
(
Fix^{+}
\gamma_{0}
)$,
that is $ \pi (L) $.
But because of the way we chose $ \gamma_{0} $,
we have that
$ (L,H) \notin W $, that is $ L \notin K $
and $ \pi (L) \notin R \cap H $.
This means that $ f(\lambda_0) \notin H $,
and so $ f $ doesn't intersect
$ Fix^{-} \gamma_{0}^{-} $
at $ \lambda_{0} $.
This means that $ f $ and 
$ Fix^{-} \gamma_{0}^{-} $
are disjoint.

\subsection{The Anosov property}
\label{sub:the_anosov_property}

Suppose that $ \Gamma $
is hyperbolic.
Suppose that the family
$
(\rho_{\lambda})
$
is irreducible at every parameter
and projectively Anosov at some parameter
$ \lambda_{0} $.
We prove that for every parameter
$ \lambda $
the representation 
$
\rho_{\lambda}
$
is projectively Anosov.

By \cite[Prop. 4.10]{guichardAnosovRepresentationsDomains2012},
an irreducible representation
$ \rho $ 
is projectively Anosov
if and only if there exists two
injective, $ \rho $-equivariant,
continuous maps
\begin{equation}
	\xi^{+}:
	\partial_{\infty} \Gamma
	\to
	\P V
	\text{ and }
	\xi^{-}:
	\partial_{\infty} \Gamma
	\to
	\P V^{*} 
	,
\end{equation}
which are:
\begin{itemize}
	\item \emph{transverse}:
		for
		$ \eta \neq \eta'
		\in
		\partial_{\infty} \Gamma $,
		the line
		$ \xi^{+}(\eta) $
		is not contained
		in the hyperplane
		$ \xi^{-}(\eta') $
	\item \emph{dynamic preserving}:
		for $ \gamma \in \Gamma $
		of infinite order,
		$ \rho(\gamma) $
		is proximal
		and
		if $ \gamma_{+} $
		is its attractive fixed point
		on
		$\partial_{\infty} \Gamma $,
		then
		$ \xi^{+}(\gamma_{+}) $
		is the attractive
		fixed point
		of $ \rho(\gamma) $ 
		and
		if $ \gamma_{-} $
		is its repulsive fixed point
		on
		$\partial_{\infty} \Gamma $,
		then
		$ \xi^{-}(\gamma_{-}) $
		is the repulsive hyperplane
		of $ \rho(\gamma) $.
\end{itemize}
These maps are called
the \emph{boundary maps}
of $ \rho $.

By hypothesis, $ \rho_{\lambda_{0}} $
is Anosov.
As being Anosov is an open property,
there exists a neighborhood $ U $
of
$ \lambda_{0} $
such that for
$ \lambda \in U $
the representation
$ \rho_{\lambda} $ 
is Anosov.
Denote the associated boundary maps
by
$ \xi^{+}_{\lambda } $
and
$ \xi^{-}_{\lambda } $.
We are going to prove
that such maps exists
at every parameter,
by defining them on
the dense set of points of the form
$ \gamma_{+} $ for $ \gamma \in \Gamma $
of infinite order
and then
extend these maps
to
$\partial_{\infty} \Gamma $.

For every $ \gamma \in \Gamma $
of infinite order,
because $ \rho_{\lambda_{0} } (\gamma) $
is proximal,
$ \gamma $ is proximal at every parameter.
By the proof of \ref{sub:fixed_points_hol}
we know that the associated graph of attractive
fixed points  is in
\g{graphs}.
Let
$ \eta = \gamma_{+}  \in \partial_{\infty} \Gamma $
be an attractive fixed point.
Define
$ f_{\eta} \in \g{graphs}  $ 
to be the graph of attractive
fixed point
of $ \gamma $
and
$ H_{\eta} \in \g{graphs_dual}  $
to be the graph of repulsive
hyperplane of
$ \gamma^{-1} $.
We have 
$
f_{\eta} 
\subset
H_{\eta} 
$.
The restriction
of
$ f_{\eta} $
to $ U $
is
$ \lambda
\mapsto
\xi^{+}_{\lambda} (\eta) $ 
and the restriction
of
$ H_{\eta} $
to $ U $
is
$ \lambda
\mapsto
\xi^{-}_{\lambda} (\eta) $.
By the transversality
of
$ \xi^{+}_{\lambda_{0}} $
and
$ \xi^{-}_{\lambda_{0}} $,
the gaphs
$ f_{\eta} $ 
and
$ H_{\eta'} $ 
don't intersect at
$ \lambda_{0} $,
for $ \eta \neq \eta' $,
so they never intersect
by lemma
\ref{lem:intersection}.
We deduce that
$ f_{\eta} $ 
and
$ f_{\eta'} $
are disjoint.

A stronger property
holds:
let
$
(f_{\eta_{n})})
$
and
$
(f_{\eta_{n}'})
$
be sequences
of such graphs
with all the
$ \eta_{n}, \eta_{n}' $ 
distincts,
such that 
$
f_{\eta_{n} }
\to f
$
and
$
f_{\eta_{n}' }
\to f'
$
.
Then
$ f $
and
$ f' $
are disjoint
or equal.
Indeed,
we have
$ f_{\eta_{n} '}
\subset
H_{\eta_{n}'}
$,
$ f_{\eta_{n}} $
and $ H_{\eta_{n} '} $
are disjoints
and we can
assume that
$
H_{\eta_{n}'}
$
converges
to some $ H' $.
By continuity of the intersection,
$ f $ 
and
$ H' $
are disjoint,
or $ f $
is contained in
$ H' $.
We can assume that
$ (\eta_{n}) $ 
converges to
$ \eta $
and
$ (\eta_{n}') $ 
converges to
$ \eta' $.
By continuity
of 
$
\xi^{+}_{\lambda} 
$
for
$ \lambda \in U $,
we have that
$ f(\lambda) $
is equal
to
$
\xi^{+}_{\lambda} 
(\eta)
$,
$ f'(\lambda) $
is equal
to
$
\xi^{+}_{\lambda} 
(\eta')
$
and
$ H'(\lambda) $ 
is equal to
$
\xi^{-}_{\lambda} 
(\eta')
$.
If $ \eta $
is different
from
$ \eta' $
then by transersality
of the boundary maps,
$ f(\lambda_{0}) $
is not contained in
$ H'(\lambda_{0}) $,
so $ f $ is
disjoint from $ H' $.
If $ \eta $
is equal to
$ \eta' $,
then $ f $ and
$ f' $
coincide
on $ U $.
By analytic continuation,
they coincide everywhere.

This strong transversality
has the following consequence.
Let $ \mathcal{F} $
be the closure of the
set of
$ f_{\eta} $ 
for
$ \eta \in \partial_{\infty} \Gamma $
an attractive fixed point.
Then for each point
$ z $ 
of
$
L(\lambda_{0})
= \xi^{+}_{\lambda_{0} }
(
\partial_{\infty} \Gamma
)
$,
there exists
exactly one
$ f \in \mathcal{F} $
such that $ f(\lambda_{0}) = z $.
Denote by
$ f_{z} $
this graph.
It is clear that
for each $ \lambda $:
\begin{equation}
	\left\{ 
		f_{z} (\lambda);
		z \in L(\lambda_{0})
	\right\}
	=
	L(\lambda)
	,
\end{equation}
and that the same property
holds:
for each $ x \in L(\lambda) $,
there exists exactly
one
$ f \in \mathcal{F} $
such that
$ f(\lambda) = x $.
Fix a $ \lambda $
and let:
\begin{align}
	\varphi :
	L(\lambda_{0}) & \to L(\lambda)
	\\
	z  & \mapsto
	f_{z} ( \lambda )
	.
\end{align}
We show that $ \varphi $
is continuous.
Let $ z_{n} \to z $
be a convergent sequence
in $ L(z_{0}) $.
Let $ l $ be
the limit
of a converging
subsequence
$ (\varphi(z_{n_{k} })) $
.
The sequence of graphs
$ (f_{z_{n_{k}}}) $
has a cluster value
$ f $.
Evaluating at $ \lambda_{0} $
we have that $ f(\lambda_{0}) = z $,
so by uniqueness,
$ f = f_{z} $.
This implies that
$ l = \varphi(z) $.
The only cluster value
of the sequence
$ (\varphi(z_{n})) $
is $ \varphi(z) $,
so 
$
\varphi(z_{n})
\to
\varphi(z)
$.
The map $ \varphi $
is continuous.

Define
$ \xi^{+}_{\lambda} $
to be
$ \varphi \circ \xi^{+}_{\lambda_{0}}$.
We construct similarly
$ \xi^{-}_{\lambda} $.
Then
$ \xi^{+}_{\lambda} $
and
$ \xi^{-}_{\lambda} $
are boundary maps
for $ \rho_{\lambda} $.
It follows that $ \rho_{\lambda} $
is Anosov.

\printbibliography
\end{document}